\documentclass[a4paper,10pt]{amsart}
\usepackage[top=0.8in, bottom=0.8in, left=1.25in, right=1.25in]{geometry}
\usepackage{fancybox}
\usepackage{amsmath}
\usepackage{amssymb}
\usepackage{amsthm}
\usepackage{mathrsfs}
\usepackage{epstopdf}
\usepackage{epstopdf}
\usepackage{array} 
\usepackage{amscd}
\usepackage{color}
\usepackage{hyperref}

\usepackage[all]{xypic}
\usepackage{tikz}

\usepackage{amsfonts}
\usepackage{amssymb}
\usepackage{mathrsfs}

\DeclareFontFamily{U}{rsfs}{%
\skewchar\font127}
\DeclareFontShape{U}{rsfs}{m}{n}{%
<-6>rsfs5<6-8.5>rsfs7<8.5->rsfs10}{}
\DeclareSymbolFont{rsfs}{U}{rsfs}{m}{n}
\DeclareSymbolFontAlphabet
{\mathrsfs}{rsfs}
\DeclareRobustCommand*\rsfs{%
\@fontswitch\relax\mathrsfs}

\theoremstyle{plain}
\newtheorem{thm}{Theorem}[section]
\newtheorem{prop}[thm]{Proposition}
\newtheorem{lem}[thm]{Lemma}

\newtheorem{defi}[thm]{Definition}

\newtheorem{rmk}[thm]{Remark}
\newtheorem{cor}[thm]{Corollary}

\newtheorem{prop-defi}[thm]{Proposition-Definition}
\newtheorem{thm-defi}[thm]{Theorem-Definition}
\newtheorem{lem-defi}[thm]{Lemma-Definition}

\newtheorem{conj}[thm]{Conjecture}

\newdimen\argwidth
\def\db[#1\db]{
 \setbox0=\hbox{$#1$}\argwidth=\wd0
 \setbox0=\hbox{$\left[\box0\right]$}
  \advance\argwidth by -\wd0
 \left[\kern.3\argwidth\box0 \kern.3\argwidth\right]}

\newcommand{\eE}{\mathcal{E}}

\newcommand{\lL}{\mathcal{L}}
\newcommand{\mM}{\mathcal{M}}

\newcommand{\oO}{\mathcal{O}}
\newcommand{\pP}{\mathcal{P}}

\newcommand{\Hom}{\mathop{\rm Hom}\nolimits}

\newcommand{\dR}{\mathbf{R}}
\newcommand{\dL}{\mathbf{L}}

\newcommand{\Hilb}{\mathop{\rm Hilb}\nolimits}

\newcommand{\Pic}{\mathop{\rm Pic}\nolimits}

\newcommand{\id}{\textrm{id}}

\newcommand{\ch}{\mathop{\rm ch}\nolimits}
\newcommand{\rk}{\mathop{\rm rk}\nolimits}
\newcommand{\td}{\mathop{\rm td}\nolimits}
\newcommand{\Ext}{\mathop{\rm Ext}\nolimits}
\newcommand{\Spec}{\mathop{\rm Spec}\nolimits}
\newcommand{\rank}{\mathop{\rm rank}\nolimits}
\newcommand{\Coh}{\mathop{\rm Coh}\nolimits}

\newcommand{\cneq}{\mathrel{\raise.095ex\hbox{:}\mkern-4.2mu=}}
\newcommand{\eqcn}{\mathrel{=\mkern-4.5mu\raise.095ex\hbox{:}}}

\newcommand{\DT}{\mathop{\rm DT}\nolimits}

\newcommand{\Imm}{\mathop{\rm Im}\nolimits}

\newcommand{\Ker}{\mathop{\rm Ker}\nolimits}

\newcommand{\RHom}{\mathop{\dR\mathrm{Hom}}\nolimits}

\makeatletter
 
  \@addtoreset{equation}{section}
\makeatother

\title[{Tautological stable pair invariants of CY 4-folds}]
{Tautological stable pair invariants \\ of Calabi-Yau 4-folds}
\date{}
\author{Yalong Cao}
\address{RIKEN Interdisciplinary Theoretical and Mathematical Sciences Program (iTHEMS), 2-1, Hirosawa, Wako-shi, Saitama, 351-0198, Japan}
\email{yalong.cao@riken.jp}
\author{Yukinobu Toda}
\address{Kavli Institute for the Physics and Mathematics of the Universe (WPI), The University of Tokyo Institutes for Advanced Study, The University of Tokyo, Kashiwa, Chiba 277-8583, Japan}
\email{yukinobu.toda@ipmu.jp}

\begin{document}
\maketitle
\begin{abstract}
Let $X$ be a Calabi-Yau 4-fold and $D$ a smooth divisor on it.   
We consider tautological complex associated with $L=\mathcal{O}_X(D)$ on the moduli space of Le Potier stable pairs and define its counting invariant by 
integrating the Euler class against the virtual class. We conjecture a formula for their generating series expressed using genus zero Gopakumar-Vafa invariants of $D$ and genus one Gopakumar-Vafa type invariants of $X$, which we verify in several examples.
When $X$ is the local resolved conifold, our conjecture reproduces a conjectural formula of Cao-Kool-Monavari in the PT chamber. In the JS chamber,
we completely determine the invariants and use it to confirm one of our previous conjectures. 
\end{abstract}



\tableofcontents

\section{Introduction}

\subsection{Background}
Gromov-Witten invariants are rational numbers that virtually count stable maps from complex curves to algebraic varieties (or symplectic manifolds).
Because of multiple-cover contributions, their enumerative meaning is a priori unclear. 
On Calabi-Yau 3-folds, motivated by string duality,  Gopakumar-Vafa \cite{GV} conjectured the existence of integral numbers (now called ``Gopakumar-Vafa invariants'') which determine Gromov-Witten invariants. This integrality conjecture was proved by Ionel-Parker \cite{IP} using methods from symplectic geometry.

On a Calabi-Yau 4-fold $X$ (i.e. a complex smooth projective 4-fold $X$ satisfying $K_X \cong \oO_X$), by virtual dimension reason, GW invariants vanish when genus $g\geqslant 2$. In \cite{KP}, Klemm and Pandharipande defined $g=0,1$ Gopakumar-Vafa type invariants (for all $\beta\in H_2(X,\mathbb{Z})$):
\begin{align}\label{intro g=0 GV}n_{0,\beta}(X)(\gamma)\in \mathbb{Q},\,\,\forall\,\, \gamma\in H^4(X,\mathbb{Z}), \quad n_{1,\beta}(X)\in \mathbb{Q},  \end{align}
in terms of Gromov-Witten invariants of $X$, and conjectured their integrality (see Section \ref{section on GV/GW} for more detail). 
Although many evidence of the integrality conjecture has been found, notably Ionel-Parker gave a proof of the genus zero integrality 
in the same paper \cite{IP}, intrinsic (e.g. sheaf theoretic) understanding of such invariants is still an interesting and important problem.

In \cite{CMT2}, Cao-Maulik-Toda gave a sheaf theoretic interpretation of \eqref{intro g=0 GV} using $\mathrm{DT}_4$ virtual classes (which are 
defined in general by Borisov-Joyce~\cite{BJ} and in special cases by ~\cite{CL1}, see also the recent work of Oh-Thomas \cite{OT} for an algebro-geometric construction) on moduli spaces of Pandharipande-Thomas (PT) stable pairs and primary insertions for $\gamma \in H^4(X, \mathbb{Z})$. The generating series of primary stable pair invariants
is conjectured to be
\begin{align}\label{intro:PTexp}
\mathrm{PT}(X)(\gamma) =\prod_{\beta} \left(
\exp(qy^{\beta})^{n_{0, \beta}(X)(\gamma)}
\cdot M(y^{\beta})^{n_{1, \beta}(X)}\right),
\end{align}
where $M(x)=\prod_{k\geqslant 1}(1-x^k)^{-k}$ is the MacMahon function (see \cite[\S0.7]{CMT2} for more details). Since primary stable pair invariants are integers, the integrality of GV type invariants 
\eqref{intro g=0 GV} is manifest from \eqref{intro:PTexp}. 
In \cite{CMT2, CKM2}, several examples were computed to support this conjecture. 
In \cite{CT1}, the authors introduced counting invariants of (Le Potier) $Z_t$-stable pairs (which recover PT stable pairs in the $t\to \infty$ limit) 
and used them to interpret \eqref{intro:PTexp} as a wall-crossing formula in the derived category of $X$. 

Apart from primary insertions, tautological insertions are usually interesting and important to consider, see for example \cite{EGL, Lehn} for works on
algebraic surfaces. In the setting of CY 4-folds, Cao-Kool \cite{CK1} considered integration of Euler classes of tautological bundles 
on virtual classes of Hilbert schemes of points and conjectured a closed formula for the generating series. 
At the same time, Nekrasov \cite{Nek} studied their 
$K$-theoretic generalization on the affine space $\mathbb{C}^4$ and found connections to physics. 
Later, Cao-Kool-Monavari \cite{CKM1} generalized such $K$-theoretic invariants to Hilbert schemes of curves and moduli spaces of PT stable pairs on 
toric CY 4-folds. Remarkably, they found a conjectural formula for the generating series of ($K$-theoretic) PT stable pair invariants on the local resolved conifold 
$\oO_{\mathbb{P}^1}(-1,-1,0)$. One motivation of this paper is to understand such a formula in the cohomological limit (see \cite[Appendix B]{CKM1}) from the perspective of global compact $\mathrm{CY_4}$ geometry.

\subsection{Our proposal}
Let $F$ be a one dimensional coherent sheaf and $s\in H^0(F)$ a section.
For an ample divisor $\omega$ on $X$, we denote the slope function
by $\mu(F)=\chi(F)/(\omega \cdot [F])$.
The pair $(F,s)$ is called $Z_t$-stable $($$t\in\mathbb{R}$$)$ if
\begin{enumerate}
\renewcommand{\labelenumi}{(\roman{enumi})}
\item for any subsheaf $0\neq F' \subset F$, we have 
$\mu(F')<t$,
\item for any
subsheaf $ F' \subsetneq F$ 
such that $s$ factors through $F'$, 
we have 
$\mu(F/F')>t$. 
\end{enumerate}
For a nonzero effective curve class $\beta \in H_2(X, \mathbb{Z})$ and $n\in \mathbb{Z}$, 
we denote by
\begin{align*}
P_n^t(X, \beta)
\end{align*}
the moduli space of 
$Z_t$-stable pairs 
$(F, s)$ with $([F], \chi(F))=(\beta, n)$. It has a wall-chamber structure and for a general $t \in \mathbb{R}$ (i.e. outside a finite subset of rational numbers in $\mathbb{R}$),  
it is a projective scheme and has a virtual class 
\begin{align}\label{intro vir cla}[P_n^t(X, \beta)]^{\mathrm{vir}}\in H_{2n}(P_n^t(X, \beta),\mathbb{Z}), \end{align}
which depends on the choice of orientation of certain real line bundle on it (Theorem \ref{existence of proj moduli space}). 

When $t<\frac{n}{\omega \cdot \beta}$, $P_n^t(X, \beta)$ is empty. The first nontrivial chamber appears when 
$t=\frac{n}{\omega \cdot \beta}+0^+$, which we call \textit{Joyce-Song (JS) chamber} (here $0^+$ denotes a sufficiently small positive number 
with respect to the fixed $\omega,\beta,n$). When $t\gg 1$, it recovers the moduli space of PT stable pairs (see Definition \ref{defi:PTJSpair} for more details).

To define tautological stable pair invariants using \eqref{intro vir cla}, we take a line bundle $L\in \Pic(X)$ and define its tautological complex:
\begin{align*}
L^{[n]}:=\dR \pi_{P\ast}(\mathbb{F}\otimes \pi_X^{\ast}L)\in \mathrm{Perf}(P^t_n(X, \beta)),
\end{align*}
where $\pi_X$, $\pi_P$ are projections from $X \times P^t_n(X,\beta)$
to the corresponding factors and 
$\mathbb{I}^{\bullet}=\{\oO \to \mathbb{F}\}$ 
is the universal $Z_t$-stable pair on $X \times P^t_n(X,\beta)$. 

The tautological $Z_t$-stable pair invariants of $X$ are defined by 
\begin{align}\label{intro taut inv}
P^t_{n,\beta}(L):=\int_{[P^t_n(X,\beta)]^{\rm{vir}}}e(L^{[n]})\in \mathbb{Z}.
\end{align}
These invariants are automatically zero unless $L\cdot \beta=0$ by the degree reason (see Section \ref{sect on tau sp inv}). 

In this paper, we are interested in the case when $L=\oO_X(D)$ corresponds to a (possibly empty) smooth divisor $D\subset X$. 
Our main conjecture is the following explicit expression of tautological stable pair invariants in terms of Gopakumar-Vafa invariants of $D$ and $X$. 
\begin{conj}\label{intro main conj}\emph{(Conjecture \ref{main conj})}
Let $(X,\omega)$ be a polarized CY 4-fold
with a projective surjective morphism 
\begin{align*}
	\pi \colon X \to B
\end{align*}
to a variety $B$. 
Let $i \colon D \hookrightarrow X$  
 be a smooth divisor of the form $D=\pi^{\ast}H$ for a 
Cartier divisor $H$ on $B$, 
and $L=\oO_X(D)$ be the associated line bundle. 
Fix a very general $t\in \mathbb{R}_{>0}$ $($i.e. outside a countable subset of rational numbers in $\mathbb{R}$$)$. 
Then for certain choice of orientation, we have 
\begin{align}\notag
	\sum_{n,\pi_*\beta=0}P^t_{n,\beta}(L)\,q^ny^{\beta} 
=\prod_{\begin{subarray}{c}\beta\in H_2(D,\mathbb{Z}) \\ \pi_*\beta=0   \end{subarray}}
\prod_{\begin{subarray}{c}1\leqslant k\leqslant [t(\omega\cdot\beta)]  \end{subarray}}(1-(-q)^ky^{i_*\beta})^{k\cdot n_{0,\beta}(D)}
\prod_{\begin{subarray}{c}\beta \in H_2(X, \mathbb{Z}) \\
		\pi_{\ast}\beta=0
\end{subarray}}M(y^{\beta})^{n_{1, \beta}(X)}. 
 \end{align}
Here $n_{0,\beta}(D)$ are genus zero GV invariants of $D$, 
$n_{1, \beta}(X)$ are genus one GV type invariants of $X$
and 
$M(x):=\prod_{k\geqslant 1}(1-x^k)^{-k}$ denotes the MacMahon function.
\end{conj}
The formulation of the 
identity
in this conjecture restricts to curve classes in fibers of $\pi$, 
so that genus zero GV invariants $n_{0, \beta}(D)$ are 
defined from moduli spaces of stable sheaves on $D$
(see Section \ref{sect on GV 3-fold}).
In most examples, curve classes on fibers are the same as curve classes which satisfy $L\cdot \beta=0$ (so there is no loss of generality restricting to 
such classes).
In contrast to $n_{0,\beta}(D)$, genus one GV type invariants $n_{1,\beta}(X)$ of $X$ in the conjectural formula are defined from Gromov-Witten theory (see Section  \ref{section on GV/GW}).

Our conjectural formula is written down by a calculation in the ideal CY 4-fold case (see Section \ref{sect on ideal case} for detail) where curves deform in families of expected
dimensions and have generic expected properties. Apart from this, we compute several examples to support our conjecture.   

\subsection{Verifications in examples}
Our computations focus on JS and PT chambers.
The first example is an elliptic fibered CY 4-fold
given by a Weierstrass model. 
\begin{thm}\emph{(Proposition \ref{prop on JS inv ell fib}, \ref{prop on PT inv ell fib})}
Let $\pi \colon  X \to \mathbb{P}^3$ be an elliptic fibered CY 4-fold \eqref{elliptic CY4}
given by a Weierstrass model and $f=\pi^{-1}(p)$ be a generic fiber. Then Conjecture \ref{intro main conj} holds for 
\begin{itemize} 
\item $n=1$, $\beta=r[f]$ $($$r\geqslant1$$)$ in JS chamber.
\item any $n$, $\beta=[f]$ in PT chamber.
\end{itemize}
\end{thm}
In the above cases, there exists a forgetful morphism 
$$P_n^t(X,\beta)\to M_n(X,\beta), \quad (s \colon \oO_X\to F)\mapsto F, $$
to the moduli scheme $M_n(X,\beta)$ of one dimensional stable sheaves $F$ with $[F]=\beta$ and $\chi(F)=n$.
By Grothendieck-Riemann-Roch formula, we reduce the computation to 
our previous computations of primary and descendent invariants on $M_n(X,\beta)$ ~\cite{CMT1, CT1}. 

Our next example is the product of a smooth projective CY 3-fold $Y$ and an elliptic curve $E$.
\begin{thm}\emph{(Theorem \ref{thm on prod of CY3}, Remark \ref{rmk on js on prod})}
Let $\pi \colon X=Y\times E\to E$ be the projection. Then Conjecture \ref{intro main conj} holds for 
\begin{itemize} 
\item $n=1$, any $\beta\in H_2(Y,\mathbb{Z})\subseteq H_2(X,\mathbb{Z})$ in JS chamber.
\item any $n$, an irreducible curve class $\beta\in H_2(Y,\mathbb{Z})\subseteq H_2(X,\mathbb{Z})$ in PT chamber
if there is an ample divisor $H$ on $Y$ such that $\mathrm{g.c.d}(n,H\cdot\beta)=1$.
\end{itemize}
\end{thm}
In these cases, there exists an isomorphism (Lemma \ref{prop on pair moduli on CY3}):
\begin{align*}P^t_n(X,\beta)\cong P^t_n(Y,\beta)\times E, 
\end{align*}
such that the virtual classes satisfy
\begin{equation}[P^t_n(X,\beta)]^{\mathrm{vir}}=[P^t_n(Y,\beta)]_{\mathrm{pair}}^{\mathrm{vir}}\otimes [E], \nonumber \end{equation}
for a choice of orientation in defining the LHS.
Here $[P^t_n(Y,\beta)]_{\mathrm{pair}}^{\mathrm{vir}}$ is the virtual class defined using the pair deformation obstruction theory, different from the construction of \cite{PT}.
Our conjecture then reduces to a statement on CY 3-folds, which we prove using Manolache's virtual push-forward formula \cite{Man}. 

Our last compact example is given by a general $(2,5)$ hypersurface $X\hookrightarrow \mathbb{P}^{1}\times \mathbb{P}^{4}$. 
By the projection to $\mathbb{P}^{1}$, it admits a quintic 3-fold fibration.
\begin{prop}\emph{(Proposition \ref{prop on quintic fib})}
Let $\pi: X\rightarrow \mathbb{P}^{1}$ be the above quintic 3-fold fibration.
Then Conjecture \ref{intro main conj} holds for any $n$ and irreducible class in PT chamber.
\end{prop}

\subsection{Local resolved conifold}
When restricted to the local resolved conifold $\oO_{\mathbb{P}^1}(-1,-1,0)$, 
one can define counting invariants using a localization formula and Conjecture \ref{intro main conj} still makes sense 
(see Section \ref{sect on local res conifold} for JS chamber and \cite{CT3} for general cases). In PT chamber, our conjecture recovers Cao-Kool-Monavari's conjectural formula in the cohomological limit (\cite[Appendix B]{CKM1}). 
In JS chamber, we completely 
determine the invariants
and prove an analogy of Conjecture \ref{intro main conj} in this setting (Theorem \ref{intro calc of JS equi inv}).

The key observation is to work with $X=\oO_{\mathbb{P}^1}(-1,-1)\times C$ for a smooth projective curve $C$. 
Although it is not necessarily a CY 4-fold, the moduli space $P^{\mathrm{JS}}_n(X,d)$ of JS stable pairs 
is a smooth projective variety of dimension $n$ (Lemma \ref{lem on smoothness of JS moduli}). 
Therefore it still makes sense to define tautological invariants. 
Note that it is enough to consider the case 
$n=kd$ for some $k \geqslant 1$, 
as $P_n^{\mathrm{JS}}(X, d)=\emptyset$ otherwise. 
A complete solution in the $n=kd$ case is given as follows: 
\begin{thm}\label{intro thm on integ for any C}\emph{(Theorem \ref{thm on integ for any C})}
Let $X=\oO_{\mathbb{P}^1}(-1,-1)\times C$ for a smooth projective curve $C$, 
and $\pi \colon X \to C$ be the projection. 
We take $L=\pi^{\ast}L_C$
for a degree $a \geqslant 0$ line bundle $L_C$ on $C$. 
Then for any fixed $k\geqslant 1$, 
we have an identity of generating series
\begin{align*}
\sum_{d\geqslant 0,\,n=kd}&\int_{P^{\mathrm{JS}}_{n}(X,d)}e(L^{[n]})\,y^d=(1+y)^{ka}.
\end{align*}
\end{thm}

An immediate consequence is: 
\begin{thm}\emph{(Theorem \ref{verify conj in JS for local P1 times E})}
Let $X=\oO_{\mathbb{P}^1}(-1,-1)\times E$ for an elliptic curve $E$ and $\pi: X\to E$ be the projection. 
Then Conjecture \ref{intro main conj} holds in JS chamber. 
\end{thm}
Another important application as mentioned above is on local resolved conifold $X=\oO_{\mathbb{P}^1}(-1,-1,0)$. As a toric CY 4-fold, there exists a CY torus $T$-action on $X$, where $T \subset (\mathbb{C}^*)^4$ is the subtorus preserving 
the CY 4-form on $X$:
\begin{align*}
\label{repara T}T=\{(t_0,t_1,t_2,t_3)\in(\mathbb{C}^*)^4:\,t_0t_1t_2t_3=1\}.
\end{align*}
The $T$-action on $X$ 
lifts to an action on $P^{\mathrm{JS}}_n(X,d)$ with finitely many reduced points as torus fixed loci. 
Therefore we can define equivariant tautological invariants (Definition \ref{defi of equi taut inv}):
$$P^{\mathrm{JS}}_{n,d}(e^m):=\sum_{\begin{subarray}{c}I=(\oO_X\to F) \in P^{\mathrm{JS}}_n(X,d)^{T} \end{subarray}}
(-1)^d\,e(\chi_X(I,I)^{\frac{1}{2}}_0+\chi_X(F)\otimes e^m)\in \frac{\mathbb{Q}(\lambda_0,\lambda_1,\lambda_2,\lambda_3,m)}{(\lambda_0+\lambda_1+\lambda_2+\lambda_3)}. $$
Here we consider a trivial $\mathbb{C}^*$-action on moduli spaces and $e^m$ is a trivial line bundle with $\mathbb{C}^*$-equivariant weight $m$, the 
equivariant Euler class is taken with respect to $T\times \mathbb{C}^*$-action. $\lambda_i=e_{T}(t_i)$'s are equivariant parameters of $T$ and 
we make an explicit choice of square root in the above definition (see \eqref{eq on choice of sqr root}). 
Again it is enough to consider the case 
$n=kd$ for $k\geqslant 1$, as 
$P_n^{\mathrm{JS}}(X, d)=\emptyset$ otherwise. 

We can classify all $T$-fixed loci and explicitly compute the invariants, though the expression involves messy combinatorics, unlike the neat formula in Conjecture \ref{intro main conj} (see Proposition \ref{JS taut invs}).
Nevertheless, we are able to relate equivariant invariants to global invariants on $\oO_{\mathbb{P}^1}(-1,-1)\times \mathbb{P}^1$ by Atiyah-Bott localization
and solve the combinatorics by using Theorem \ref{intro thm on integ for any C}:
\begin{thm}\label{intro calc of JS equi inv}\emph{(Theorem \ref{calc of JS equi inv})}
Let $X=\oO_{\mathbb{P}^1}(-1,-1,0)$ and $-\lambda_3$ be the equivariant parameter of $\oO_{\mathbb{P}^1}(0)$. Then for any fixed $k\geqslant 1$, 
we have an identity of generating series
\begin{align*}\sum_{d\geqslant 0}&P^{\mathrm{JS}}_{kd, d}(e^m)y^d=(1-y)^{k\cdot\frac{m}{-\lambda_3}}. \end{align*}
\end{thm}
By taking certain limit of $m$, the above equivariant tautological invariants recover equivariant invariants without insertions, which allows us to 
prove one of our previous conjectures:
\begin{cor}\emph{(Corollary \ref{cor on inv no insert})}
We have 
\begin{align*}
\sum_{\begin{subarray}{c}I \in P^{\mathrm{JS}}_n(X,d)^{T} \end{subarray}}e_{T}(\chi_X(I,I)^{\frac{1}{2}}_0)=
\left\{\begin{array}{rcl}\frac{1}{d\,!\,\lambda_3^d} \quad         &\mathrm{if} \,\,n=d, \\
& \\ 0   \, \, \,\, \,  \quad      & \,\,  \mathrm{otherwise}, 
\end{array} \right. 
\end{align*}
i.e. Conjecture 6.10 of \cite{CT1} holds. 
\end{cor} 

\subsection{Acknowledgement}
Y. C. is grateful to Martijn Kool and Sergej Monavari for previous collaboration \cite{CKM1} without which the current paper would not exist. 
We thank the referee for very careful reading of our paper and providing many helpful comments which improves the exposition of this paper. 
This work is partially supported by the World Premier International Research Center Initiative (WPI), MEXT, Japan.
Y. C. is partially supported by RIKEN Interdisciplinary Theoretical and Mathematical Sciences
Program (iTHEMS), JSPS KAKENHI Grant Number JP19K23397 and Royal Society Newton International Fellowships Alumni 2019 and 2020.
Y. T. is partially supported by Grant-in Aid for Scientific Research grant (No. 26287002) from MEXT, Japan.

\section{Definitions and conjectures}

\subsection{GV type invariants of CY 4-folds}\label{section on GV/GW}
Let $X$ be a Calabi-Yau 4-fold.
The genus 0 Gromov-Witten invariants on $X$ are defined using
insertions: for $\gamma \in H^{4}(X, \mathbb{Z})$,
one defines
\begin{equation}
\mathrm{GW}_{0, \beta}(\gamma)
=\int_{[\overline{M}_{0, 1}(X, \beta)]^{\rm{vir}}}\mathrm{ev}^{\ast}(\gamma)\in\mathbb{Q},
\nonumber \end{equation}
where $\mathrm{ev} \colon \overline{M}_{0, 1}(X, \beta)\to X$
is the evaluation map.

The \textit{genus 0 Gopakumar-Vafa type invariants} 
\begin{align}\label{g=0 GV}
n_{0, \beta}(\gamma) \in \mathbb{Q}
\end{align}
are defined by Klemm-Pandharipande \cite{KP} from the identity
\begin{align*}
\sum_{\beta>0}\mathrm{GW}_{0, \beta}(\gamma)q^{\beta}=
\sum_{\beta>0}n_{0, \beta}(\gamma) \sum_{d=1}^{\infty}
\frac{1}{d^{2}}\,q^{d\beta}.
\end{align*}
For the genus 1 case, virtual dimensions of moduli spaces of stable maps are zero, so Gromov-Witten invariants
\begin{align*}
\mathrm{GW}_{1, \beta}=\int_{[\overline{M}_{1, 0}(X, \beta)]^{\rm{vir}}}
1 \in \mathbb{Q}
\end{align*}
can be defined
without insertions.
The \textit{genus 1 Gopakumar-Vafa type invariants}
\begin{align}\label{g=1 GV}
n_{1, \beta} \in \mathbb{Q}
\end{align}
 are defined in~\cite{KP} by the identity
\begin{align*}
\sum_{\beta>0}
\mathrm{GW}_{1, \beta}q^{\beta}=
&\sum_{\beta>0} n_{1, \beta} \sum_{d=1}^{\infty}
\frac{\sigma(d)}{d}q^{d\beta}
+\frac{1}{24}\sum_{\beta>0} n_{0, \beta}(c_2(X))\log(1-q^{\beta}) \\
&-\frac{1}{24}\sum_{\beta_1, \beta_2}m_{\beta_1, \beta_2}
\log(1-q^{\beta_1+\beta_2}),
\end{align*}
where $\sigma(d)=\sum_{i|d}i$ and $m_{\beta_1, \beta_2}\in\mathbb{Z}$ are called meeting invariants defined in~\cite[Section~0.3]{KP}. 

In~\cite{KP}, both of the invariants (\ref{g=0 GV}), (\ref{g=1 GV})
are conjectured to be integers, and GW invariants are computed to support the conjectures in many examples.
Note that the genus zero integrality conjecture has been proved 
by Ionel-Parker \cite[Theorem 9.2]{IP} using symplectic geometry.

\subsection{$Z_t$-stable pairs} 

Let $\omega$ be an ample divisor on $X$ and $t\in\mathbb{R}$, we recall the following notion of $Z_t$-stable pairs.
\begin{defi}\label{def Zt sta}\emph{(\cite[Lemma 1.7]{CT1})}
Let $F$ be a one dimensional coherent sheaf and $s: \oO_X\to F$ be a section. For an ample divisor $\omega$, we denote the slope function
by $\mu(F)=\chi(F)/(\omega \cdot [F])$.

We say $(F,s)$ is a $Z_t$-(semi)stable pair $($$t\in\mathbb{R}$$)$ if 
\begin{enumerate}
\renewcommand{\labelenumi}{(\roman{enumi})}
\item for any subsheaf $0\neq F' \subset F$, we have 
$\mu(F')<(\leqslant)t$,
\item for any
subsheaf $ F' \subsetneq F$ 
such that $s$ factors through $F'$, 
we have 
$\mu(F/F')>(\geqslant)t$. 
\end{enumerate}
\end{defi}
There are two distinguished stability conditions appearing as 
$Z_t$-stability in some chambers. 
\begin{defi}\label{defi:PTJSpair}\emph{(\cite{PT}, \cite[Definition 1.10]{CT1})} 

(i) A pair $(F,s)$ is a PT stable pair if
$F$ is a pure one dimensional sheaf and $s$ is surjective in dimension one. 

(ii) A pair $(F,s)$ is a JS stable pair if $s$ is a non-zero morphism, $F$ is $\mu$-semistable and 
for any subsheaf $0\neq F' \subsetneq F$ such that $s$ factors through 
$F'$ we have $\mu(F')<\mu(F)$. 
\end{defi}
\begin{prop}\label{prop:chambers}\emph{(\cite[Proposition 1.11]{CT1})} 
For a pair $(F,s)$ with $[F]=\beta$ and $\chi(F)=n$, its

(i) $Z_t$-stability with $t\to \infty$ is exactly PT stability, 

(ii) $Z_t$-stability with $t=\frac{n}{\omega\cdot \beta}+0^+$ is exactly JS stability. 
\end{prop}
Let $\beta \in H_2(X, \mathbb{Z})$ and $n\in \mathbb{Z}$, we denote 
$$P^t_n(X, \beta)\subseteq \pP^t_n(X, \beta) $$
to be the moduli stack of $Z_t$-stable (semistable) pairs $(F,s)$ with $[F]=\beta$ and $\chi(F)=n$.

By Proposition \ref{prop:chambers}, there are two disinguished moduli spaces, 
PT moduli spaces and JS moduli spaces, 
by specializing $t\to \infty$ and $t=\frac{n}{\omega\cdot \beta}+0^+$ respectively:
\begin{align*}
P_n(X, \beta) \cneq P_n^{t\to \infty}(X, \beta), \quad
P_n^{\mathrm{JS}}(X, \beta) \cneq 
P_n^{t=\frac{n}{\omega\cdot \beta}+0^+}(X, \beta). 
\end{align*}

As $Z_t$-stable pairs are special cases of Le Potier's stable coherent systems \cite{LePotier} whose moduli spaces can be constructed by GIT, so 
$P^t_n(X, \beta)$ is a quasi-projective scheme, and $\pP^t_n(X, \beta)$ admits a good moduli space
\begin{align*}
\pP^t_n(X, \beta) \to \overline{P}_n^t(X, \beta),
\end{align*}
where $\overline{P}_n^t(X, \beta)$ is a projective 
scheme which parametrizes $Z_t$-polystable objects.
The following result shows moduli stacks of $Z_t$-stable pairs are indeed open substacks of moduli stacks of objects in the derived categories of coherent sheaves.
\begin{thm}\label{existence of proj moduli space}\emph{(\cite[Theorem 0.1]{CT1})} 
$P^t_n(X, \beta)$ admits an open immersion 
$$P^t_n(X, \beta)\to \mathcal{M}_0, \quad (F,s)\mapsto (\oO_X\stackrel{s}{\to} F) $$
to the moduli stack $\mathcal{M}_0 $ of $E\in D^b\Coh (X)$ with $\Ext^{<0}(E,E)=0$ and $\det(E)\cong \oO_X$.

Therefore for a general choice $($i.e. outside a finite subset of rational numbers in $\mathbb{R}$$)$\,of $t$, $P^t_n(X, \beta)$ is a projective scheme which has 
a virtual class
\begin{align}\label{pair vir class}
[P^t_n(X, \beta)]^{\rm{vir}}\in H_{2n}\big(P^t_n(X, \beta),\mathbb{Z}\big), 
\end{align}
in the sense of Borisov-Joyce \cite{BJ}. The virtual 
class depends on the choice of orientation of certain (real) line bundle over $P^t_n(X, \beta)$ \cite{CGJ, CL2}. 
\end{thm}
In \cite{CMT2, CT1}, the authors gave sheaf theoretic interpretations of GV type invariants \eqref{g=0 GV}, \eqref{g=1 GV} using integration of primary insertions on
\eqref{pair vir class}. See also \cite{CMT1, CT2} for related sheaf theoretic interpretations using different moduli spaces.

\subsection{Tautological stable pair invariants} \label{sect on tau sp inv}
The $\DT_4$ tautological invariants are considered in the case of Hilbert schemes of points in \cite{CK1}. We consider their 
extension to moduli spaces of $Z_t$-stable pairs.
For a line bundle $L\in \Pic(X)$, we define its tautological complex:
\begin{align*}
L^{[n]}:=\dR \pi_{P\ast}(\mathbb{F}\otimes \pi_X^{\ast}L)\in 
\mathrm{Perf}(P^t_n(X, \beta)),
\end{align*}
where $\pi_X$, $\pi_P$ are projections from $X \times P^t_n(X,\beta)$
to corresponding factors and 
$$
\mathbb{I}^{\bullet}=\{\oO \to \mathbb{F}\}
$$ 
is the universal $Z_t$-stable pair on $X \times P^t_n(X,\beta)$. 

The tautological $Z_t$-stable pair invariants of $X$ are defined by 
\begin{align}\label{taut inv}
P^t_{n,\beta}(L):=\int_{[P^t_n(X,\beta)]^{\rm{vir}}}e(L^{[n]}).
\end{align}
Here for a $K$-theory class $\eE=A-B$ (where $A,B$ are vector bundles), its Euler class is
$$e(\eE):=\left[\frac{c_{\bullet}(A)}{c_{\bullet}(B)}\right]_{\rk(\eE)}, $$
where $c_{\bullet}$ denotes the total Chern class.

In JS and PT chambers, we write  
$$P_{n,\beta}(L):=P^{t\to \infty}_{n,\beta}(L) , \quad P^{\mathrm{JS}}_{n,\beta}(L):=P^{t=\frac{n}{\omega\cdot \beta}+0^+}_{n,\beta}(L).$$
By the Riemann-Roch formula, we have
$$\rank (L^{[n]})=(n+L\cdot \beta), $$ 
so invariants \eqref{taut inv} vanish unless $L\cdot \beta=0$.

\subsection{Genus 0 GV invariants of 3-folds}\label{sect on GV 3-fold}
Let $D$ be a smooth projective 3-fold which admits a projective morphism 
\begin{align}\label{fib of D}\pi: D\to H \end{align}
to a variety $H$. We assume $K_D=\pi^*\mathcal{L}$ is the pullback of a line bundle $\mathcal{L}\in \Pic(H)$. 

Let $M_1(D,\beta)$ be the coarse moduli scheme of one dimensional 
stable sheaves $E$ on $D$ with $[E]=\beta$ and $\chi(E)=1$. For curve classes in fibers of $\pi$, we can construct a virtual class.
\begin{prop}\label{prop on vir class on D}
If $\pi_*\beta=0$, then there exists a virtual class 
$$[M_1(D,\beta)]^{\mathrm{vir}}\in A_0(M_1(D,\beta),\mathbb{Z}). $$
in the sense of Behrend-Fantechi and Li-Tian \cite{BF, LT}.
\end{prop}
\begin{proof}
For curve class $\beta$ such that $\pi_*\beta=0$, any $[E]\in M_1(D,\beta)$ is scheme theoretically supported on a fiber of $\pi$ \eqref{fib of D} (ref. \cite[Lemma 2.2]{CMT1}).
Therefore, we have  
$$E\otimes K_D\cong E\otimes \pi^*\mathcal{L}\cong E.$$
By Serre duality, we obtain 
$$\Ext^3(E,E)\cong \Hom(E,E\otimes K_D)^{\vee}\cong \Hom(E,E)^{\vee}=\mathbb{C}. $$
Therefore one can truncate the obstruction theory and define the virtual class.
Since $\chi(E, E)=0$ by the Riemann-Roch theorem, the 
virtual dimension is zero. 
\end{proof}
Using Proposition \ref{prop on vir class on D}, we can define genus zero Gopakumar-Vafa invariants of $D$:
\begin{align}\label{def of g=0 inv on D}n_{0,\beta}(D):=\int_{[M_1(D,\beta)]^{\mathrm{vir}}}1\in \mathbb{Z}, \quad \mathrm{if}\,\, \pi_*\beta=0. \end{align}
When $D$ is a CY 3-fold, we can take \eqref{fib of D} to be the constant map to a point. Our definition of genus zero GV invariants then 
reduces to the definition of Katz \cite{Katz}.

\subsection{Conjecture}
The following is our main conjecture of this paper which expresses tautological $Z_t$-stable pair invariants in terms of 
genus zero GV invariants of the associated 3-folds and genus one GV type invariants of CY 4-folds.

\begin{conj}\label{main conj} 
	Let $(X,\omega)$ be a polarized CY 4-fold
	with a projective surjective morphism 
	\begin{align*}
		\pi \colon X \to B
	\end{align*}
	to a variety $B$. 
	Let $i \colon D \hookrightarrow X$
	be a smooth divisor of the form $D=\pi^{\ast}H$ for a 
	Cartier divisor $H$ on $B$, 
	and $L=\oO_X(D)$ be the associated line bundle. 
	Fix a very general $t\in \mathbb{R}_{>0}$. 
	Then for certain choice of orientation, we have 
	\begin{align}\notag
		\sum_{n,\pi_*\beta=0}P^t_{n,\beta}(L)\,q^ny^{\beta} 
		=\prod_{\begin{subarray}{c}\beta\in H_2(D,\mathbb{Z}) \\ \pi_*\beta=0   \end{subarray}}
		\prod_{\begin{subarray}{c}1\leqslant k\leqslant [t(\omega\cdot\beta)]  \end{subarray}}(1-(-q)^ky^{i_*\beta})^{k\cdot n_{0,\beta}(D)}
		\prod_{\begin{subarray}{c}\beta \in H_2(X, \mathbb{Z}) \\
				\pi_{\ast}\beta=0
		\end{subarray}}M(y^{\beta})^{n_{1, \beta}(X)}. 
	\end{align}
	Here $n_{0,\beta}(D)$ are genus zero GV invariants of $D$ \eqref{def of g=0 inv on D}, 
	$n_{1, \beta}(X)$ are genus one GV type invariants of $X$ \eqref{g=1 GV}
	and 
	$M(x):=\prod_{k\geqslant 1}(1-x^k)^{-k}$ denotes the MacMahon function.
\end{conj}
The formulation of the 
identity in this conjecture restricts to curve classes in fibers of $\pi$, 
so that genus zero GV invariants $n_{0, \beta}(D)$ are 
defined from moduli spaces of stable sheaves on $D$.
In most examples, curve classes on fibers are the same as curve classes which satisfy $L\cdot \beta=0$.
In contrast to $n_{0,\beta}(D)$, genus one GV type invariants $n_{1,\beta}(X)$ of $X$ in the conjectural formula are defined from Gromov-Witten theory 
(see Section \ref{section on GV/GW}).


\subsection{Heuristic argument in ideal geometry}\label{sect on ideal case}
In this section, modulo the issue of choosing orientations, we justify Conjecture \ref{main conj} in the following ideal case. 

Let $X$ be an `ideal' CY 4-fold
in the sense that all curves of $X$ deform in families of expected dimensions, and have expected generic properties, i.e.
\begin{enumerate}
\item
any rational curve in $X$ is a chain of smooth $\mathbb{P}^1$ with normal bundle $\mathcal{O}_{\mathbb{P}^{1}}(-1,-1,0)$, and
moves in a compact 1-dimensional smooth family of embedded rational curves, whose general member is smooth with 
normal bundle $\mathcal{O}_{\mathbb{P}^{1}}(-1,-1,0)$. 
\item
any elliptic curve $E$ in $X$ is smooth, super-rigid, i.e. 
the normal bundle is 
$L_1 \oplus L_2 \oplus L_3$
for general degree zero line bundle $L_i$ on $E$
satisfying $L_1 \otimes L_2 \otimes L_3=\oO_E$. 
Furthermore any two elliptic curves are 
disjoint and also disjoint with rational curve families.
\item
there is no curve in $X$ with genus $g\geqslant 2$.
\end{enumerate}
In this ideal case, geometric objects in $X$ should be in general positions and one does not expect $X$ to have a fibration structure. 
We justify Conjecture \ref{main conj} in a more general situation, where $D\subset X$ is a general smooth divisor and we consider all curve classes $\beta$ which satisfy $L\cdot \beta=0$. 

We first consider the contribution of rational curve families to tautological invariants.
We only need to consider curve classes $\beta$ such that $L\cdot \beta=0$. 
\begin{lem}
For a $Z_t$-stable pair $(F,s)$ supported on rational curve families, we have 
$$H^1(F\otimes L)=H^1(F)=0. $$
Therefore, the tautological complex $L^{[n]}$ is a vector bundle. 
\end{lem}
\begin{proof}
By the $Z_t$-stability, $\oO_Z:=\mathrm{Im}(s)$ is a non-zero subsheaf of $F$. If $\oO_Z\neq F$, we have 
$$1\leqslant \mu(\oO_Z)<t, \quad \frac{\chi(F)-\chi(\oO_Z)}{d(F)-d(\oO_Z)}>t, $$
where $d(F)$ denotes the degree of the curve class $[F]$.
This implies 
$$ 1\leqslant \mu(\oO_Z)< \mu(F). $$
Consider the Harder-Narasimhan/Jordan-H\"older filtration: 
$$0=F_0\subset F_1\subset \cdots \subset F_l=F, $$
where $F_i/F_{i-1}$'s are stable one dimensional sheaves with non-increasing slopes. 

If $l>1$, consider surjection $F\twoheadrightarrow F/F_{l-1}$, then $Z_t$-stability implies 
$$\mu(F/F_{l-1})>t>1.$$
Therefore all $F_i/F_{i-1}$'s have positive slope. They are also stable and scheme theoretically supported on some rational curve, 
so $H^1(F)=0$ by the generic normal bundle assumption.

If $l=1$, then $F$ is stable with positive slope, we also have vanishing $H^1(F)=0$.
The vanishing $H^1(F \otimes L)=0$ also holds similarly. 
\end{proof}
\begin{lem}\label{lem on tat sect}
There exists a tautological section $\sigma$ of vector bundle $L^{[n]}$ whose zero locus satisfies 
\begin{align}\label{id:subscheme}
Z(\sigma)=P^t_{n}(D,\beta), 
\end{align}
as closed subschemes of $P^t_{n}(X,\beta)$.
\end{lem}
\begin{proof}
Let $\psi: \oO_{P^t_{n}(X,\beta)\times X}\to \mathbb{F}$ be the universal stable pair and consider the diagram
\begin{displaymath}
\xymatrix
{&  \mathbb{F}  \ar_{}[d] & \\
& P^t_{n}(X,\beta)\times X \ar_p[ld] \ar^q[dr]   & \\
P^t_{n}(X,\beta) & & X.
}
\end{displaymath}
Denote $s:\oO_X\to L=\oO_X(D)$ to be the section which defines $D\subset X$. The composition 
$$\oO_{P^t_{n}(X,\beta)\times X}\stackrel{\psi}{\to}  \mathbb{F} \stackrel{q^*s}{\to}  \mathbb{F}\otimes q^*L $$
gives a tautological section $\sigma$ of $L^{[n]}$.
We claim the desired equality (\ref{id:subscheme}) of subschemes 
of $P_n^t(X, \beta)$. 

In order to see this, we take 
a morphism $f \colon T \to P_n^t(X, \beta)$
which corresponds to a pair 
$(\mathbb{F}_T, \psi_T)$
where $\mathbb{F}_T$ is a 
$T$-flat family 
\begin{displaymath}
\xymatrix
{&  \mathbb{F}_T  \ar_{}[d] & \\
& T\times X \ar_{p_T}[ld] \ar^{q_T}[dr]   & \\
T & & X.}
\end{displaymath}
Let $\iota \colon T \times D \hookrightarrow T \times X $
be the inclusion and $\mathbb{F}_{T_D}=\iota^*\mathbb{F}_T$. We need to show $\mathbb{F}_T=\iota_*\mathbb{F}_{T_D}$ 
if and only if the morphism $f \colon
 T \rightarrow P^t_{n}(X,\beta)$ factors through $Z(\sigma)$. 

Now $f$ factors through $Z(\sigma)$ if and only if $f^*\sigma$ is the zero section of $f^*L^{[n]}$.  Note that
$$f^* \sigma = f^*(p_*(q^*s\cdot \psi))=p_{T*}(f^*(q^*s\cdot \psi))=p_{T*}( q_T^{*}s\cdot\psi_T), $$
where $\psi_T \colon \oO\to \mathbb{F}_T$ is the universal pair on $T\times X$.
This is a zero section if and only if for any $t\in T$ which corresponds to $(F,\phi)$, the composition 
$$\oO_X\stackrel{\phi}{\to} F\stackrel{s_F}{\to} F\otimes L$$
is zero (here $s_F$ denotes the multiplication by section $s$). This is equivalent to
 $\phi$ factors through $\Ker(s_F)\subseteq F$. We claim that $\mathrm{Im}(s_F)=0$, otherwise violating the stability of $(F,\phi)$. In fact, if $\Ker(s_F)\neq F$, by $Z_t$-stability,   
$$\mu(\mathrm{Im}(s_F))>t. $$
At the same time, for the nonzero subsheaf $\mathrm{Im}(s_F)\otimes L^{-1}\subseteq F$, we have 
$$\mu(\mathrm{Im}(s_F)\otimes L^{-1})<t. $$
As we only consider curve classes in rational curve families whose intersection with $L$ is trivial, so the curve class of $\mathrm{Im}(s_F)$ satisfies 
$[\mathrm{Im}(s_F)]\cdot L=0$. Thus 
$$\mu(\mathrm{Im}(s_F)\otimes L^{-1})=\mu(\mathrm{Im}(s_F)), $$
which gives a contradiction. 
Therefore we have showed $\phi$ factors through $\Ker(s_F)\subseteq F$ if and only if $\mathrm{Im}(s_F)=0$, i.e. $F$ sits in $D$ scheme theoretically. Therefore we are done.
\end{proof}
When the section $\sigma$ is `transverse' to the zero section in the following diagram:
\begin{align*}\xymatrix{
  & L^{[n]} \ar[d]_{\pi}  &  &  \\
 P^t_{n}(D,\beta) \ar[r]^{ \iota} & P^t_{n}(X,\beta),  \ar@/_1pc/[u]_{\sigma} &  &  }
\end{align*}
heuristically, we should have 
\begin{equation}\label{eq integral 1}\int_{[P^t_{n}(X,\beta)]^{\mathrm{vir}}}e(L^{[n]})``="\int_{[P^t_{n}(D,\beta)]_{\mathrm{ind}}^{\mathrm{vir}}}1, \end{equation}
where $[P^t_{n}(D,\beta)]_{\mathrm{ind}}^{\mathrm{vir}}$ is some degree zero cycle in $P^t_{n}(D,\beta)$ 
induced from $P^t_{n}(X,\beta)$.

Generically, the divisor $D$ intersects each rational curve family in a finite number of disjoint smooth rational curves and 
we may ignore intersections of different rational curve families. We only consider curve classes $\beta$ such that $L\cdot \beta=0$.
To evaluate the LHS of \eqref{eq integral 1}, we may reduce to a computation on
$$X_0=\oO_{\mathbb{P}^1}(-1,-1,0), \quad D_0=\oO_{\mathbb{P}^1}(-1,-1)\times \{p\}, $$
where the irreducible curve class in $X_0$ is a class $\beta$ in the ambient compact CY 4-fold $X$.

Though $X_0$ is non-compact, the RHS of \eqref{eq integral 1} still makes sense as the zero locus of the tautological section $\sigma$ is 
the moduli space of stable pairs on $D_0$ which is proper. We apply Calabi-Yau torus $T\cong (\mathbb{C}^*)^3$-localization on $X_0$ and 
arrive at equivariant invariants
\begin{align*}\sum_{\begin{subarray}{c}I=(\oO_{X_0}\to F) \in P^{t(\omega\cdot\beta)}_n(X_0,d\,[\mathbb{P}^1])^{T} \end{subarray}}
e_{T}(\chi_{X_0}(I,I)^{\frac{1}{2}}_0)\cdot e_{T}(\chi_{X_0}(F)\otimes L)&
=\sum_{\begin{subarray}{c}I \in P^{t(\omega\cdot\beta)}_n(D_0,d\,[\mathbb{P}^1])^{T} \end{subarray}}e_{T}(\chi_{D_0}(I,I)_0),
\end{align*}
where $L=\oO_{X_0}(D_0)$.
Here we note that the polarization $\omega$ 
restricts to 
 $\omega|_{X_0}=\oO(\omega\cdot\beta)$ 
 on $X_0$, since its degree on the zero section $\mathbb{P}^1$ is $(\omega\cdot\beta)$. 

By the dimensional reduction result of \cite[Proposition~2.6]{CT3}, \cite{CKM1},
only stable pairs which are scheme theoretically supported on $D$ contribute to the invariants, which explains the above equality.
The generating series of the RHS  
is known by the work of Nagao-Nakajima \cite[Theorem~3.12]{NN}:
$$\sum_{n,d}\sum_{\begin{subarray}{c}I \in P^{t(\omega\cdot\beta)}_n(D_0,d\,[\mathbb{P}^1])^{T} \end{subarray}}e_{T}(\chi_{D_0}(I,I)_0)\,q^ny^d
=\prod_{k=1}^{[t(\omega\cdot\beta)]}(1-(-q)^k y)^{k}, $$
where $[t(\omega\cdot\beta)]$ denotes the largest integer not bigger than $t(\omega\cdot\beta)$.

Now we go back to the global picture. 
Although $D\subset X$ may not be a CY 3-fold, we only consider curve classes $\beta$ in $D$ such that $K_{D}\cdot \beta=0$, so $D$ behaves 
like CY for such curve classes. The intersection number of $D$ with all rational curve families in classes $\beta$ with $L\cdot \beta=0$ 
is exactly the number of $(-1,-1)$ curves (i.e.~smooth rational curves whose normal bundle is $\oO_{\mathbb{P}^1}(-1,-1)$) in $D$ in the ideal situation, i.e. equal to the genus zero GV invariant of $D$. 
To sum up, all rational curve families contribute 
$$\prod_{\begin{subarray}{c} \beta\in H_2(D,\mathbb{Z}) \\ K_D\cdot \beta=0 \end{subarray}}
\prod_{\begin{subarray}{c}1\leqslant k\leqslant [t(\omega\cdot\beta)]  \end{subarray}}
(1-(-q)^ky^{i_*\beta})^{k\cdot n_{0,\beta}(D)}$$
to the generating series of tautological $Z_t$-stable pair invariants in Conjecture \ref{main conj}. 

Next we consider contribution from super-rigid elliptic curves. In this case, it suffices to consider the total space
$$X_0=\mathrm{Tot}_E(L_1\oplus L_2 \oplus L_3)$$ 
of direct sum of three general degree zero line bundles on an elliptic curve $E$ satisfying $L_1 \otimes L_2 \otimes L_3=\oO_E$. As we only need to consider $L$ with $L\cdot \beta=0$, so we may assume $L=\oO$ when restricted to $X_0$. 

When the Euler characteristic is zero, we have
$$P^t_{0,d[E]}(\oO)=\int_{[P^t_0(X_0,d[E])]^{\rm{vir}}}1, $$
whose generating series is computed in \cite[Section 4.1]{CT1}, given by the MacMahon function:
$$M(y):=\prod_{k\geqslant 1}(1-y^k)^{-k}. $$
Summing over all possible super-rigid elliptic curves, they contribute 
$$\prod_{\begin{subarray}{c}\beta\in H_2(X,\mathbb{Z})   \end{subarray}}M(y^\beta)^{n_{1,\beta}(X)}$$
to the generating series of tautological $Z_t$-stable pair invariants (with $L=\oO$) in Conjecture \ref{main conj}. 

When the Euler characteristic $n$ is positive, we claim that
\begin{align}\label{equ on van}P^t_{n,d[E]}(\oO)=0, \quad n>0.\end{align}
This vanishing is obvious if the tautological complex $\oO^{[n]}$ is a bundle, 
as it admits a nowhere vanishing section by the argument in Lemma \ref{lem on tat sect}.
In JS chamber, this indeed happens:
\begin{prop}\label{lem on vec bdl on JS}
Let $X=\mathrm{Tot}_E(L_1\oplus L_2 \oplus L_3)$ be as above and $d,n\in \mathbb{Z}_{\geqslant 1}$. 
Then the tautological complex $\oO^{[n]}$ over $P^{\mathrm{JS}}_n(X,d)$ is a vector bundle. Therefore \eqref{equ on van} holds in JS chamber.
\end{prop}
\begin{proof}
It is enough to show $H^1(X,F)=0$ for any $(F,s)\in P^{\mathrm{JS}}_n(X,d)$. We take the Jordan-H\"older filtration of $F$: 
$$0=F_0\subset F_1\subset\cdots \subset F_l=F, $$
such that $F_i/F_{i-1}$'s are stable with the same slope. We claim that stable one dimensional sheaves on $X$ are scheme theoretically supported on 
the zero section $i:E\hookrightarrow X$ and hence
$$H^1(X,F_i/F_{i-1})\cong H^1(E,\pi_*(F_i/F_{i-1}))\cong \Hom_E(\pi_*(F_i/F_{i-1}),\oO_E)^{\vee}=0, $$
where $\pi \colon X\to E$ is the projection and the last equality is because $\pi_*(F_i/F_{i-1})$ is stable with positive slope. By induction,
we know $H^1(X,F)=0$. 

Therefore, we are left to prove our claim that any stable one dimensional sheaf $\eE$ on $X$ is scheme theoretically supported on $E$. 
We first show 
\begin{align}\label{equ on vani of hom}\Hom_X(\eE\otimes \pi^*L_i^{-1},\eE)=0, \quad \forall \,\, i=1,2,3.  \end{align}
As $\eE\otimes \pi^*L_i^{-1}$ and $\eE$ are stable with the same slope, if the above Hom space is not zero, then
$$\eE\cong \eE\otimes \pi^*L_i^{-1}. $$ 
Applying $\pi_*$ and taking determinant, we know $L_i\in \Pic^0(E)$ is a torsion line bundle which is a contradiction. 
Finally, consider the Koszul resolution
$$\cdots \to \pi^*(L_1^{-1}\oplus L_2^{-1}\oplus L_3^{-1})\to \oO_X\to \oO_E\to 0, $$
and tensor it with $\eE$, we obtain an exact sequence 
$$\eE\otimes \pi^*(L_1^{-1}\oplus L_2^{-1}\oplus L_3^{-1}) \to \eE \to \eE\otimes \oO_E\to 0. $$
By \eqref{equ on vani of hom}, we conclude that $\eE\cong \eE\otimes \oO_E$, so we are done.
\end{proof}
In other chambers, $\oO^{[n]}$ is in general not a bundle and it seems more tricky to show \eqref{equ on van}.
We believe when a wall-crossing formula is established, the vanishing will be manifestly reduced to the vanishing in JS chamber (wall-crossing formulae in $\DT_4$ setting are not established yet, see \cite{GJT} for a recent 
conjectural proposal).
In Appendix \ref{sect on vanishing}, we prove \eqref{equ on van} in PT chamber when $d=2$.


\section{Examples} 
When $n=0$, Conjecture \ref{main conj} reduces to our previous conjectures \cite[Conjecture 0.2]{CT1}, \cite[Conjecture 0.2]{CMT2}, and evidence
is given in loc. cit. In this section, we concentrate on the case of $n>0$ and compute several examples to support our conjecture.
\subsection{Elliptic fibrations}
For $Y=\mathbb{P}^3$, we take general elements
\begin{align*}
u \in H^0(Y, \oO_Y(-4K_Y)), \quad 
v \in H^0(Y, \oO_Y(-6K_Y)).
\end{align*}
Let $X$
be a Calabi-Yau 4-fold with an elliptic fibration
\begin{align}\label{elliptic CY4}
\pi \colon X \to Y
\end{align}
given by the equation
\begin{align*}
zy^2=x^3 +uxz^2+vz^3
\end{align*}
in the $\mathbb{P}^2$-bundle
\begin{align*}
\mathbb{P}(\oO_Y(-2K_Y) \oplus \oO_Y(-3K_Y) \oplus \oO_Y) \to Y,
\end{align*}
where $[x:y:z]$ are homogeneous coordinates for the above
projective bundle. A general fiber of
$\pi$ is a smooth elliptic curve, and any singular
fiber is either a nodal or cuspidal plane curve.
Moreover, $\pi$ admits a section $\iota$ whose image
corresponds to the fiber point $[0: 1: 0]$.

Let $h$ be a hyperplane in $\mathbb{P}^3$, $f$ be a general fiber of $\pi \colon X\rightarrow Y$ and set
\begin{align}\label{div:BE}
B=:\pi^{\ast}h, \ D:=\iota(\mathbb{P}^3)\in H_{6}(X,\mathbb{Z}).
\end{align}
\subsubsection{JS chamber}
Let $M_{1}(X, r[f])$ be the coarse moduli space of 
one dimensional stable sheaves $E$ on $X$ with $[E]=r[f]$ and $\chi(E)=1$.
The moduli space $P_1^{\mathrm{JS}}(X,r[f])$ of JS stable pairs has a forgetful map 
\begin{align}\label{forget map 1}\psi: P_1^{\mathrm{JS}}(X,r[f])\stackrel{\cong}{\to} M_{1}(X,r[f]), \quad (\oO_X \stackrel{s}{\to} F) \mapsto F,  \end{align}
which is an isomorphism, under which virtual classes can be identified (\cite[Proposition 5.4]{CT1}).
\begin{lem}$($\cite[Lemma 2.1]{CMT1}$)$\label{lem:fib:isom}
For any $r \in \mathbb{Z}_{\geqslant 1}$, there is an isomorphism
$$M_{1}(X,r[f]) \cong X, $$
under which the virtual
class of $M_{1}(X,r[f])$ is given by
\begin{align*} 
[M_{1}(X,r[f])]^{\rm{vir}}=\pm \mathrm{PD}(c_3(X)) \in H_2(X, \mathbb{Z}),
\end{align*}
where the sign corresponds to a choice of orientation in defining the LHS.
\end{lem}

\begin{prop}\label{prop on JS inv ell fib}
Let $\pi \colon X \to \mathbb{P}^3$ be the elliptic fibered CY 4-fold \eqref{elliptic CY4}, $S\subset \mathbb{P}^3$ be a general degree $l$ hypersurface and $L=\oO_X(\pi^{-1}(S))$ be the line bundle. For multiple 
fiber classes $r[f]$ with $r\geqslant 1$, we have  
$$P^{\mathrm{JS}}_{1,r[f]}(L)=\pm n_{0,r[f]}(\pi^{-1}(S))=\pm 960\,l, $$
i.e. Conjecture \ref{main conj} holds in JS chamber for this case.
\end{prop}
\begin{proof}
We first compute invariants on the 3-fold. The method is similar to \cite[Lemma 2.1]{CMT1}. 
Since $W:=\pi^{-1}(S)$ has an elliptic fibration to $S$, by the Fourier-Mukai transform \cite{BM}, we have 
$$M_{1}(W,r[f])\cong W, $$
under which the virtual class satisfies 
$$[M_{1}(W,r[f])]^{\mathrm{vir}}=\mathrm{PD}(e(\wedge^2 TW))=\mathrm{PD}(-c_3(W)+c_1(W)c_2(W)). $$
Using the exact sequence 
$$0\to TW\to TX|_W\to K_W\to 0, $$
we have 
$$\int_{X}c_3(X)\cdot c_1(\oO_X(W))=\int_{W}c_3(W)-c_1(W)c_2(W). $$
Therefore, 
$$n_{0,r[f]}(W)=-\int_{X}c_3(X)\cdot c_1(\oO_X(W))=-l\int_{X}c_3(X)\cdot B=960\,l, $$
where we use \cite[(61)]{KP} in the last equality. 

Next we compute JS stable pair invariants. Denote 
$$\{\oO_{P_1^{\mathrm{JS}}(X,r[f])\times X}\stackrel{s}{\to} \mathbb{F}\} $$ 
to be the universal JS stable pair.
From this, we get a nowhere zero tautological section 
$$\oO_{P_1^{\mathrm{JS}}(X,r[f])}\stackrel{\pi_{P\ast}(s)}{\longrightarrow}\dR\pi_{P\ast}(\mathbb{F}), $$
of the tautological bundle $\dR\pi_{P\ast}(\mathbb{F})$, where we use the fact that $H^1(F)=0$ for any $(\oO_X\to F)\in P_1^{\mathrm{JS}}(X,r[f])$ and 
$\pi_P \colon
 P_1^{\mathrm{JS}}(X,r[f])\times X\to P_1^{\mathrm{JS}}(X,r[f])$ denotes the projection map. Taking determinant of this section gives an isomorphism 
$$\det\left(\dR\pi_{P\ast}(\mathbb{F})\right)\cong \oO_{P_1^{\mathrm{JS}}(X,r[f])}. $$ 
This means under the isomorphism $\psi$ \eqref{forget map 1}, $\mathbb{F}$ is the pullback of the `normalized' universal one dimensional sheaf 
$\mathbb{F}_{\mathrm{norm}}$ of $M_{1}(X,r[f])$ introduced in \cite[(1.7)]{CT2}.

Let $\pi_M$, $\pi_X$ be projections of $M_1(X,r[f])\times X$ to each factor. 
For $L=\oO_X(\pi^{-1}(S))$, the tautological bundle $\dR\pi_{M\ast}(\mathbb{F}_{\mathrm{norm}}\boxtimes L)$ is rank one. 
By Grothendieck-Riemann-Roch formula, 
\begin{align*}P^{\mathrm{JS}}_{1,r[f]}(L)
&=\int_{[M_{1}(X,r[f])]^{\mathrm{vir}}}e\left(\dR\pi_{M\ast}(\mathbb{F}_{\mathrm{norm}}\boxtimes L)\right)\\
&=\int_{[M_{1}(X,r[f])]^{\mathrm{vir}}}\ch\left(\dR\pi_{M\ast}(\mathbb{F}_{\mathrm{norm}}\boxtimes L)\right) \\
&=\int_{[M_{1}(X,r[f])]^{\mathrm{vir}}}\pi_{M\ast}\big(\ch(\mathbb{F}_{\mathrm{norm}})\cdot\pi_{X}^*(\ch(L)\cdot \td(X))\big) \\
&=\int_{[M_{1}(X,r[f])]^{\mathrm{vir}}}\pi_{M\ast}\left(\ch_3(\mathbb{F}_{\mathrm{norm}})\cdot\pi_{X}^*\left(\frac{1}{2}c_1(L)^2+\frac{1}{12}c_2(X)\right)\right) \\
&\quad +\int_{[M_{1}(X,r[f])]^{\mathrm{vir}}}\pi_{M\ast}\left(\ch_4(\mathbb{F}_{\mathrm{norm}})\cdot\pi_{X}^*\left(c_1(L)\right)\right)
+\int_{[M_{1}(X,r[f])]^{\mathrm{vir}}}\pi_{M\ast}\left(\ch_5(\mathbb{F}_{\mathrm{norm}})\right) \\
&=\int_{[M_{1}(X,r[f])]^{\mathrm{vir}}}\pi_{M\ast}\left(\ch_4(\mathbb{F}_{\mathrm{norm}})\cdot\pi_{X}^*\left(c_1(L)\right)\right)=\pm 960\,l.
\end{align*}
Here we use Lemma \ref{lem:fib:isom} in the second equality and in the last equality we use previous computations of primary/descendent invariants of $M_{1}(X,r[f])$ (ref. \cite[Proposition 2.3]{CMT1}, \cite[Proposition 2.7, Section 1.7]{CT2}).
\end{proof}

\subsubsection{PT chamber}
For the fiber class $[f]$, we have a similar forgetful map 
\begin{align}\label{forget map 2}\psi: P_n(X,[f])\to M_{n}(X,[f]), \quad (\oO_X \stackrel{s}{\to} F) \mapsto F, \quad n\geqslant 1, \end{align}
which is a $\mathbb{P}^{n-1}$-bundle, where $M_{n}(X, [f])$ is the coarse moduli space of 
one dimensional stable sheaves $E$ on $X$ with $[E]=[f]$ and $\chi(E)=n$.
\begin{lem}\label{lem identify Mn and M1} 
We have an isomorphism
\begin{align*}
M_{1}(X,[f]) \stackrel{\cong}{\to} M_{n}(X,[f]), \quad E\mapsto E\otimes \oO_X(D)^{\otimes (n-1)}, 
\end{align*}
and the virtual classes of both sides are identified 
under the above isomorphism. 
\end{lem}
\begin{proof}
Any $E\in M_{n}(X, [f])$ is scheme theoretically supported on a fiber of \eqref{elliptic CY4} (ref. \cite[Lemma 2.2]{CMT1}), 
so $E\otimes \oO_X(D)$ is stable with $[E\otimes \oO_X(D)]=[f]$ and $\chi(E\otimes \oO_X(D))=n+1$. 
Therefore we obtain the above isomorphism. Because of this isomorphism, $M_{n}(X,[f])$ is a fine moduli 
space whose virtual class can be defined as in $M_{1}(X,[f])$ case. 
The identification of virtual classes is obvious.
\end{proof}
\begin{lem}\label{lem on PT vir class ell fib}
Let $n\geqslant 1$. Under the morphism \eqref{forget map 2}, we have 
\begin{align*} 
[P_{n}(X,[f])]^{\rm{vir}}=\psi^*[M_{n}(X,[f])]^{\rm{vir}} \in H_{2n}(P_{n}(X,[f]), \mathbb{Z}),
\end{align*}
for a choice of orientation.
\end{lem}
\begin{proof}
Let $I=(\oO_X\to F)\in P_{n}(X,[f])$, then $F$ is stable and scheme theoretically supported on a fiber. By Serre duality, we have 
$H^1(F)=0$. 
By \cite[Proof of Lemma 2.3]{CMT2}, we have distinguished triangles
\begin{align*}
&\RHom_X(F,F)\to \RHom_X(\oO_X,F) \to \RHom_X(I,F), \\
&\RHom_X(I, F) \to \RHom_X(I, I)_0[1] \to \RHom_X(F, \oO_X)[2].
\end{align*}
By taking the associated long exact sequences of cohomologies 
and using $H^1(F)=0$, we have the isomorphisms  
\begin{align*}\label{ell fib 2}
\Ext^2_X(F,F)\stackrel{\cong}{\leftarrow}
 \Ext^1_X(I,F) \stackrel{\cong}{\to}
 \Ext^2_X(I,I)_0,   \end{align*}
under which Serre duality can be identified. A family version of this argument shows 
the obstruction bundle of $P_{n}(X,[f])$ is the pull-back of the obstruction bundle of $M_{n}(X,[f])$.
\end{proof}
Similarly to Proposition \ref{prop on JS inv ell fib}, we have 
the following proposition: 
\begin{prop}\label{prop on PT inv ell fib}
Let $S\subset \mathbb{P}^3$ be a general degree $l$ hypersurface and $\pi^{-1}(S)$ be the inverse image under \eqref{elliptic CY4}. 
Then for $L=\oO_X(\pi^{-1}(S))$ and $n\geqslant 1$, we have 
\begin{align*}
P_{n,[f]}(L)=\pm n\cdot n_{0,[f]}(\pi^{-1}(S))=\pm 960\,ln, 
\end{align*}
i.e. Conjecture \ref{main conj} holds in PT chamber for this case.
\end{prop}
\begin{proof}
We have the Cartesian diagram:
\begin{align*}\xymatrix{
 X & \ar[l]_{\pi_X\quad \quad \, } P_{n}(X,[f])\times X \ar[d]_{\pi_P} \ar[r]^{\bar{\psi}} & M_{n}(X,[f])\times X \ar[d]_{\pi_M} \ar[r]^{\quad \quad \quad \pi_X} & X  &  \\
&  P_{n}(X,[f]) \ar[r]^{\psi} & M_{n}(X,[f]).  &  &  }
\end{align*}

Let $\mathbb{F} \to M_{n}(X,[f]) \times X$ be a universal sheaf. 
Then the map $\psi$ \eqref{forget map 2} identifies $P_n(X, [f])$ with 
$\mathbb{P}(\pi_{M\ast}\mathbb{F})$. The universal stable pair is given by 
\begin{align*}
\mathbb{I}=
(\mathcal{O}_{P_n(X,[f])\times X} 
\stackrel{s}{\rightarrow} \mathbb{F}^{\dag}), \quad
\mathbb{F}^{\dag}:=\bar{\psi}^{\ast}\mathbb{F}\otimes \oO(1),
\end{align*} 
where $\oO(1)$ is the tautological line bundle on 
$\mathbb{P}(\pi_{M\ast}\mathbb{F})$
and $s$ is the tautological map.  

For a line bundle $L\to X$ such that $L\cdot [f]=0$, we have
\begin{align*}P_{n,[f]}(L) 
&=\int_{[P_{n}(X,[f])]^{\mathrm{vir}}}e\left(\dR\pi_{P\ast}(\bar{\psi}^*\mathbb{F}\otimes  \pi_X^*L)\otimes \oO(1)\right)\\
&=\int_{[P_{n}(X,[f])]^{\mathrm{vir}}}e\left(\psi^*\dR\pi_{M\ast}(\mathbb{F}\otimes  \pi_X^*L)\otimes \oO(1)\right)\\
&=\int_{P_{n}(X,[f])}e\left(\psi^*\dR\pi_{M\ast}(\mathbb{F}\otimes  \pi_X^*L)\otimes \oO(1)\right)\cdot \psi^*e(\mathrm{Ob}_{M_{n}(X,[f])})\\
&=\int_{P_{n}(X,[f])}c_1(\oO(1))^{n-1}c_1\left(\psi^*\dR\pi_{M\ast}(\mathbb{F}\otimes  \pi_X^*L)\right)\cdot \psi^*e(\mathrm{Ob}_{M_{n}(X,[f])})\\
&\quad +\int_{P_{n}(X,[f])}c_1(\oO(1))^{n}\cdot \psi^*e(\mathrm{Ob}_{M_{n}(X,[f])}) \\
&=\int_{P_{n}(X,[f])}c_1(\oO(1))^{n-1}c_1\left(\psi^*\dR\pi_{M\ast}(\mathbb{F}\otimes  \pi_X^*L)\right)\cdot \psi^*e(\mathrm{Ob}_{M_{n}(X,[f])})\\
&\quad -\int_{P_{n}(X,[f])}c_1(\oO(1))^{n-1}c_1\left(\psi^*\dR\pi_{M\ast}(\mathbb{F})\right)\cdot \psi^*e(\mathrm{Ob}_{M_{n}(X,[f])}) \\
&=\int_{[M_{n}(X,[f])]^{\mathrm{vir}}}\Big(c_1\left(\dR\pi_{M\ast}(\mathbb{F}\otimes  \pi_X^*L)\right)-c_1\left(\dR\pi_{M\ast}(\mathbb{F})\right)\Big) \\
&=\int_{[M_{1}(X,[f])]^{\mathrm{vir}}}c_1\left(\dR\pi_{M\ast}\left(\mathbb{F}_1\otimes \pi_X^*\left(\oO_X(D)^{\otimes(n-1)}\otimes L\right)\right)\right) \\
&\quad -\int_{[M_{1}(X,[f])]^{\mathrm{vir}}}c_1\left(\dR\pi_{M\ast}\left(\mathbb{F}_1\otimes \pi_X^*\oO_X(D)^{\otimes(n-1)}\right)\right).  
\end{align*}
Here in the second equality we use the base-change \cite[Lemma 1.3]{BO}, in the third equality we denote $\mathrm{Ob}_{M_{n}(X,[f])}$ to be the obstruction bundle of $M_{n}(X,[f])$ and use Lemma \ref{lem on PT vir class ell fib}, in the fourth equality we use the fact that the obstruction bundle is rank three and $M_{n}(X,[f])\cong X$ (Lemma \ref{lem:fib:isom} and Lemma \ref{lem identify Mn and M1}), in the fifth equality we use the relation in the cohomology ring of the projective bundle 
$P_{n}(X,[f])=\mathbb{P}(\pi_{M\ast}\mathbb{F})$, in the final equality we denote $\mathbb{F}_1$ to be a universal sheaf of $M_{1}(X,[f])$ and 
use Lemma \ref{lem identify Mn and M1}.

The RHS of the above last equality is independent of the choice of $\mathbb{F}_1\to M_{1}(X,[f])\times X$, we may take it to be the normalized one $\mathbb{F}_{\mathrm{norm}}$, i.e. 
$\det\left(\dR\pi_{M\ast}(\mathbb{F}_{\mathrm{norm}})\right)\cong \oO$. Then 
\begin{align*}
&\quad \, \int_{[M_{1}(X,[f])]^{\mathrm{vir}}}c_1\left(\dR\pi_{M\ast}\left(\mathbb{F}_{\mathrm{norm}}\otimes \pi_X^*\left(\oO_X(D)^{\otimes(n-1)}\otimes L\right)\right)\right) \\
&=\int_{[M_{1}(X,[f])]^{\mathrm{vir}}}\pi_{M\ast}\left(\ch_4(\mathbb{F}_{\mathrm{norm}})\cdot\pi_{X}^*\left(c_1(L)+(n-1)D\right)\right)
+\int_{[M_{1}(X,[f])]^{\mathrm{vir}}}\pi_{M\ast}\left(\ch_5(\mathbb{F}_{\mathrm{norm}})\right)  \\
&\quad +\int_{[M_{1}(X,[f])]^{\mathrm{vir}}}\pi_{M\ast}\left(\ch_3(\mathbb{F}_{\mathrm{norm}})\cdot\pi_{X}^*\left(\frac{(c_1(L)+(n-1)D)^2}{2}+\frac{1}{12}c_2(X)\right)\right). 
\end{align*}
Take $L=\oO_X(\pi^{-1}S)$, $\oO_X$, and consider the difference:
\begin{align*}
&\quad \,\, \int_{[M_{1}(X,[f])]^{\mathrm{vir}}}c_1\left(\dR\pi_{M\ast}\left(\mathbb{F}_{\mathrm{norm}}\otimes \pi_X^*\left(\oO_X(D)^{\otimes(n-1)}\otimes \oO_X(\pi^{-1}S)\right)\right)\right) \\
&\quad \, -\int_{[M_{1}(X,[f])]^{\mathrm{vir}}}c_1\left(\dR\pi_{M\ast}\left(\mathbb{F}_{\mathrm{norm}}\otimes \pi_X^*\left(\oO_X(D)^{\otimes(n-1)}\right)\right)\right) \\
&=\int_{[M_{1}(X,[f])]^{\mathrm{vir}}}\left[\pi_{M\ast}\left(\ch_4(\mathbb{F}_{\mathrm{norm}})\cdot\pi_{X}^*\left(lB\right)\right)+\pi_{M\ast}\left(\ch_3(\mathbb{F}_{\mathrm{norm}})\cdot\pi_{X}^*\left(\frac{(lB)^2+2(n-1)lB\cdot D}{2}\right)\right)\right] \\
&=\pm(960\,l+960(n-1)l)=\pm 960\,nl,
\end{align*}
where $B,D$ are the divisors in \eqref{div:BE} and in the last equality use the previous computations of primary and descendent invariants \cite[Proposition 2.3]{CMT1}, \cite[Proposition 2.7]{CT2}.
\end{proof}
 
\subsection{$\mathrm{CY_3}\times E$}
Let $X=Y\times E$ be the product of a smooth projective CY 3-fold $Y$ with an elliptic curve $E$.
This is a trivial elliptic fibration over $Y$. 
For multiple (elliptic) fiber classes, 
we have similar results as Proposition \ref{prop on JS inv ell fib} and Proposition \ref{prop on PT inv ell fib}.
Take $X$ as a trivial CY 3-fold fibration to $E$ and consider curve classes from fibers, we recall the following:
\begin{lem}\label{prop on pair moduli on CY3}\emph{(\cite[Proposition 2.11]{CMT2})}
For an irreducible curve class $\beta\in H_2(Y,\mathbb{Z})\subseteq H_2(X,\mathbb{Z})$, we have an isomorphism 
\begin{equation}P_n(X,\beta)\cong P_n(Y,\beta)\times E. \nonumber \end{equation}
The virtual class of $P_n(X,\beta)$ satisfies 
\begin{equation}[P_n(X,\beta)]^{\mathrm{vir}}=[P_n(Y,\beta)]_{\mathrm{pair}}^{\mathrm{vir}}\otimes [E], \nonumber \end{equation}
for certain choice of orientation in defining the LHS.
Here $[P_n(Y,\beta)]_{\mathrm{pair}}^{\mathrm{vir}}\in A_{n-1}(P_n(Y,\beta),\mathbb{Z})$ is the virtual class defined using 
a truncation of the pair deformation obstruction theory.
\end{lem}
Having this, it is easy to see:
\begin{lem} 
Let $\beta\in H_2(Y,\mathbb{Z})\subseteq H_2(X,\mathbb{Z})$ be an irreducible curve class, $n\geqslant 1$ and $L=\oO_X(Y\times \{p_1,\ldots,p_l\})$. Then
Conjecture \ref{main conj} holds in PT chamber for $\beta,n,L$ if and only if:
\begin{align}\label{equ on 3-fold pair vir class}\int_{[P_n(Y,\beta)]_{\mathrm{pair}}^{\mathrm{vir}}}c_{n-1}\left(\pi_{P\ast}\mathbb{F}\right)=\pm\,n\cdot n_{0,\beta}(Y), 
\end{align}
where $(\oO_{P_n(Y,\beta)\times Y}\to \mathbb{F})$ is the universal pair and $\pi_P \colon P_n(Y,\beta)\times Y\to P_n(Y,\beta)$ is the projection.
\end{lem}
We prove the required equality on CY 3-folds by Manolache's virtual push-forward formula.
\begin{thm}\label{thm on prod of CY3}
The equality \eqref{equ on 3-fold pair vir class} holds 
for any CY 3-fold $Y$ if there is an ample divisor $H$ on $Y$ such that $\mathrm{g.c.d}(n,H\cdot\beta)=1$. 
\end{thm}
\begin{proof}
Let $M_{n}(Y, \beta)$ be the coarse moduli space of 
one dimensional stable sheaves $F$ on $Y$ with $[F]=\beta$ and $\chi(F)=n$. Since $\beta$ is irreducible, we have (e.g. \cite[Lemma 2.1]{Toda3}):
$$n_{0,\beta}(Y)=\int_{[M_{n}(Y, \beta)]^{\mathrm{vir}}}1. $$
There is a forgetful morphism 
\begin{align*}g: P_n(Y,\beta)\to M_{n}(Y,\beta), \quad (\oO_X \stackrel{s}{\to} F) \mapsto F,  \end{align*}
whose fiber over $F$ is $\mathbb{P}(H^0(Y,F))$.
By the assumption, there is a universal sheaf $\mathbb{F} \to M_{n}(Y,\beta) \times Y$.
The map $g$ identifies $P_n(Y, \beta)$ with 
$\mathbb{P}(\pi_{M\ast}\mathbb{F})$,
where $\pi_M \colon M_{n}(Y, \beta) \times Y \to M_{n}(Y, \beta)$
is the projection. The universal stable pair is then given by 
\begin{align*}
\mathbb{I}=
(\mathcal{O}_{Y\times P_n(Y, \beta)} 
\stackrel{s}{\rightarrow} \mathbb{F}^{\dag}), \quad
\mathbb{F}^{\dag}:= (\id_Y \times f)^{\ast}\mathbb{F}\otimes \oO(1),
\end{align*} 
where $\oO(1)$ is the tautological line bundle on 
$\mathbb{P}(\pi_{M\ast}\mathbb{F})$
and $s$ is the tautological map.  

As in the proof of \cite[Proposition 2.10]{CMT2}, the map $g$
has a relative perfect obstruction theory and satisfies Manolache's virtual push-forward formula: 
$$g_*\big(c_{n-1}(\pi_{P\ast}\mathbb{F})\cdot [P_n(Y,\beta)]_{\mathrm{pair}}^{\mathrm{vir}}\big)=c\cdot [M_{n}(Y, \beta)]^{\mathrm{vir}}, $$
for each connected component of $M_{n}(Y, \beta)$. 
To fix $c$, it is enough to restrict the relative obstruction theory to a fiber of $g$ (ref. \cite[pp. 2022, (18)]{Man}). The obstruction bundle 
over $\mathbb{P}(H^0(Y,F))$ is $H^1(Y,F)\otimes \oO(1)$. Adding the tautological insertion, we have 
\begin{align*}c&=\int_{\mathbb{P}(H^0(Y,F))}e(H^1(Y,F)\otimes \oO(1))\cdot c_{n-1}(\chi(Y,F)\otimes \oO(1)) \\
&= \int_{\mathbb{P}(H^0(Y,F))}H^{h^1(Y,F)}\cdot \frac{(1+H)^{h^0(Y,F)}}{(1+H)^{h^1(Y,F)}}\bigg|_{\deg (n-1)} \\
&= \int_{\mathbb{P}(H^0(Y,F))}H^{h^1(Y,F)}\cdot (1+H)^{n}|_{\deg (n-1)} \\
&= n\cdot \int_{\mathbb{P}(H^0(Y,F))}H^{h^1(Y,F)}\cdot H^{n-1}=n,
\end{align*}
where $H$ denotes the hyperplane class of $\mathbb{P}(H^0(Y,F))$.
Therefore,
\begin{align*}\int_{[P_n(Y,\beta)]_{\mathrm{pair}}^{\mathrm{vir}}}c_{n-1}\left(\pi_{P\ast}\mathbb{F}\right)=n\cdot \int_{[M_{n}(Y, \beta)]^{\mathrm{vir}}}1=
n\cdot n_{0,\beta}(Y). \quad \quad \quad   \qedhere   \end{align*} 
\end{proof}
\begin{rmk}\label{rmk on js on prod}
By the same argument, one can similarly show Conjecture \ref{main conj} holds in JS chamber for $n=1$, any 
$\beta\in H_2(Y,\mathbb{Z})\subseteq H_2(X,\mathbb{Z})$ and $L=\oO_X(Y\times \{p_1,\ldots,p_l\})$.
\end{rmk}

\subsection{Quintic fibrations}
Let $X\hookrightarrow \mathbb{P}^{1}\times \mathbb{P}^{4}$ be a general $(2,5)$ hypersurface. By the projection to $\mathbb{P}^{1}$, it admits a quintic 3-fold fibration:
\begin{equation}\label{qui fib}\pi \colon X\rightarrow \mathbb{P}^{1},  \end{equation}
i.e. $\pi$ is a proper morphism whose general fiber is a smooth quintic 3-fold $Y\subseteq\mathbb{P}^{4}$.
We consider an irreducible curve class $\beta\in H_2(X,\mathbb{Z})$ such that $\pi_*\beta=0$, which is the class of a line $l$ in a generic quintic fiber.
\begin{lem}\label{lem on quin fib}
For a generic $(2,5)$ hypersurface $X\hookrightarrow \mathbb{P}^{1}\times \mathbb{P}^{4}$, 
the Hilbert scheme $\Hilb(X,[l])$ of lines $\{p\}\times l\subset  \mathbb{P}^{1}\times \mathbb{P}^{4}$ which sits in $X$ is smooth of one dimension.
\end{lem} 
\begin{proof}
This is a standard argument following \cite[Theorem 4.3, pp. 266]{Kol} (see also \cite[Proposition 1.4]{Caoconic}). 
Let $G:=\mathbb{P}^{1}\times \mathrm{Gr}(2,5)$ denote the Hilbert scheme of lines $\{p\}\times l\subset  \mathbb{P}^{1}\times \mathbb{P}^{4}$,
and $P:=\mathbb{P}(H^0(\mathbb{P}^{1}\times \mathbb{P}^{4},\oO(2,5)))=\mathbb{P}^{491}$ be the projective space of all 
$(2,5)$ hypersurfaces in $\mathbb{P}^{1}\times \mathbb{P}^{4}$. Consider the incident subscheme
$$\mathcal{I}:=\{(C,X)\in G\times P:\, C\subset X\}. $$
For projection $\pi_1: \mathcal{I}\to G$, the surjectivity of the restriction map 
$$H^0(\mathbb{P}^{1}\times \mathbb{P}^{4},\oO(2,5))\to H^0(C,\oO(2,5)|_C) $$
shows $\mathcal{I}$ is irreducible smooth of dimension $378$. Consider the second projection $\pi_2: \mathcal{I}\to P$, by the 
generic smoothness, there exists a non-empty open subset $U\subset P$ such that a fiber of $\pi_2|_{\pi_2^{-1}(U)}$
is smooth of dimension one.
\end{proof}
\begin{prop}\label{prop on quintic fib}
Let $Y=\pi^{-1}(p)$ be a general fiber and $L=\oO_{X}(Y)$.
Then for a choice of orientation, we have 
\begin{align*}P_{n,[l]}(L)=(-1)^{n+1}\,n\cdot n_{0,[l]}(Y)=(-1)^{n+1}\,2875\,n,  \quad \forall \,\, n\geqslant 0, \end{align*}
i.e. Conjecture \ref{main conj} holds in PT chamber for this case.
\end{prop}
\begin{proof}
For a general quintic fiber $Y$, there are exactly 2875 lines $l$'s on it. For each $l$,  
$$N_{l/Y}\cong \oO_{\mathbb{P}^1}(-1,-1), \quad N_{l/X}\cong \oO_{\mathbb{P}^1}(-1,-1,0). $$
By Lemma \ref{lem on quin fib}, such lines deform to all fibers of $\pi$.
Let $M_n(X,[l])$ denote the moduli scheme of one dimensional stable sheaves $E$ with $[E]=[l]$ and $\chi(E)=n$.
Any such $E$ is scheme theoretically 
supported on a fiber of $\pi$, so $M_n(X,[l])$ has a fibration structure: 
$$M_n(X,[l])\to \mathbb{P}^1, $$
where fibers parametrize  $2875$ lines in fibers of $\pi$. 

As in \eqref{forget map 2}, by forgetting the section, we have a $\mathbb{P}^{n-1}$-bundle structure 
$$\psi: P_n(X,[l])\to M_n(X,[l]). $$
Therefore, $P_n(X,[l])$ is smooth of dimension $n$ whose virtual class
is the usual fundamental class for a choice of orientation.

The tautological bundle $L^{{[n]}}$ has a tautological section $\sigma$ whose zero locus is $P_n(Y,[l])$ (see Lemma \ref{lem on tat sect}).
We have a morphism 
of vector bundles
$$T_{P_{n}(X,[l])}|_{P_{n}(Y,[l])}\stackrel{d\sigma}{\to} L^{[n]}|_{P_{n}(Y,[l])}, $$
whose kernel is $T_{P_{n}(Y,[l])}$. We claim that 
\begin{align*}\mathrm{Coker}(d\sigma)\cong T_{P_{n}(Y,[l])}, \end{align*}
where the proof will be given later, following the proof of Theorem \ref{thm on integ for any C}. Therefore
\begin{align*}
\int_{P_n(X,[l])}e(L^{[n]})=\int_{P_n(Y,[l])}e(\mathrm{Coker}(d\sigma))=\chi(P_n(Y,[l]))=2875\chi(\mathbb{P}^{n-1})
=2875n.
\end{align*}
By choosing certain orientation of the virtual class, we can put sign $(-1)^{n+1}$ for the number.
\end{proof}

\subsection{Local resolved conifold}

For $X=\oO_{\mathbb{P}^1}(-1,-1,0)$, although moduli spaces $P^t_n(X,\beta)$ of $Z_t$-stable pairs are non-compact, one can define 
similar tautological invariants using torus localization \cite{CKM1, CT3}.
In \cite[Appendix B]{CKM1}, a similar closed formula as in Conjecture \ref{main conj} is conjectured in PT chamber and verified for certain orders by a
vertex calculation (see also \cite[Section 2.9]{CK2} for the case without insertions). 
In \cite{CT3}, we extend such a story to general $Z_t$-stable pairs (and also `non-commutative chamber'). 
In fact, this paper is motivated by the formula in loc. cit., which we want to understand from the perspective of (global) compact CY 4-folds.
In Section \ref{section JS pair on local P1}, we completely determine the invariants in JS chamber.

\section{Joyce-Song stable pairs on $\oO_{\mathbb{P}^1}(-1,-1)\times C$}\label{section JS pair on local P1}
Let $C$ be a smooth projective curve, and 
take $X$ to be
\begin{align}\label{XOC}
X=\oO_{\mathbb{P}^1}(-1,-1)\times C.
\end{align}
In this section, we completely determine
tautological JS stable pair invariants on $X$
and apply it to prove our previous conjecture~\cite[Conjecture~6.10]{CT1}
on equivariant JS stable pair invariants (without insertions) on 
local resolved conifold $\oO_{\mathbb{P}^1}(-1, -1, 0)$. 

\subsection{Tautological invariants}
Let $X$ be a 4-fold defined by (\ref{XOC}), 
and $\pi \colon X \to \mathbb{P}^1$ be the projection. 
For a compactly supported one dimensional sheaf $F$ on $X$,
its slope is defined by 
\begin{align*}
\mu(F) \cneq \frac{n}{d} \in \mathbb{Q}, \quad 
\left([\pi_{\ast}F], \chi(F)\right)=\left(d\,[\mathbb{P}^1], n\right).
\end{align*}
For $(d, n) \in \mathbb{Z}^2$,
we consider the moduli space of JS stable pairs: 
\begin{align*}
P_n^{\mathrm{JS}}(X, d)
=\big\{(F, s) : \mbox{it is a JS stable pair with }
([\pi_{\ast}F], \chi(F))=(d\,[\mathbb{P}^1], n)\big\}. 
\end{align*} 
Here in our non-compact setting, the 
JS stability for $(F, s)$
is defined by: 
\begin{itemize}
\item $F$ is a compactly supported one dimensional semistable sheaf, 
\item $s \neq 0$, and for any subsheaf $\Imm(s) \subset F' \subsetneq F$
we have $\mu(F')<\mu(F)$. 
\end{itemize}
The following lemma shows that $P^{\mathrm{JS}}_n(X, d)$ has a nice 
geometric property:
\begin{lem}\label{lem on smoothness of JS moduli}
The moduli space 
$P^{\mathrm{JS}}_n(X, d)$ is empty unless $n=(k+1)d$ for some $k\geqslant 0$. 
For $n=(k+1)d$, 
the moduli space 
$P^{\mathrm{JS}}_n(X,d)$ is a smooth projective variety of dimension $n$.
\end{lem}
\begin{proof}
Let 
 \begin{align*}
i \colon \mathbb{P}^1\times C \hookrightarrow \oO_{\mathbb{P}^1}(-1,-1)\times C
\end{align*}
be the product of the zero section with the identity map on $C$.
By taking the Jordan-H\"older filtration, it is easy to see that any 
one dimensional semistable sheaf $F$ on $X$ with 
$[\pi_{\ast}F]=d[\mathbb{P}^1]$ is of the form 
$i_{\ast}(\oO_{\mathbb{P}^1}(k) \boxtimes Q)$
for a zero dimensional sheaf $Q$ on $C$ with length $d$. 
Such a sheaf has a non-zero section if and only if $k \geqslant 0$. 
Therefore $P_n^{\mathrm{JS}}(X, d)$ is empty unless 
$n=(k+1)d$ for some $k\geqslant 0$. 

Let $\pP air_n(X, d)$ be the moduli stack of pairs 
$(F, s)$, 
where $F$ is a one dimensional semistable sheaf with 
$([\pi_{\ast}F], \chi(F))=(d\,[\mathbb{P}^1], n)$ and 
$s \colon \oO_X \to F$ is a section (without JS stability). 
We have the following diagram 
\begin{align}\label{dia:Pair}
\xymatrix{ P^{\mathrm{JS}}_n(X,d) \ar@<-0.3ex>@{^{(}->}[r]^-{j} \ar[rd]
 & \pP air_n(X,d) \ar[d]^{p} \\ & 
\mM_{n}(X,d).  }
\end{align}
Here 
$j$ is a natural open immersion, 
$\mM_n(X, d)$ is the moduli stack of 
 one dimensional semistable sheaves $F$ with 
$([\pi_{\ast}F], \chi(F))=(d\,[\mathbb{P}^1], n)$ and 
$p$ is the forgetful morphism. 
The morphism $p$ is an affine space bundle with fiber $H^0(F)=\mathbb{C}^n$ and $p\circ j$ is a smooth morphism.

Let $\mM_d(C)$ be the moduli stack of length $d$ zero dimensional 
sheaves on $C$. 
We have an isomorphism of stacks 
\begin{align*}
\mM_{d}(C)\stackrel{\cong}{\to} \mM_{n}(X,d), \quad 
Q\mapsto i_{\ast}( \oO_{\mathbb{P}^1}(k) \boxtimes Q ).
\end{align*}
Since $\Ext_C^2(Q, Q)=0$ and $\chi(Q, Q)=0$
for a zero dimensional sheaf $Q$ on $C$, 
the 
stack $\mM_d(C)$ is smooth of dimension zero. 
Therefore from the diagram (\ref{dia:Pair}), 
we conclude that 
$P^{\mathrm{JS}}_n(X,d)$ is smooth of dimension $n$. 

Let $\oO_{\mathbb{P}^1}(-1, -1) \subset P$ be a projective compactification 
and set $\overline{X}=P \times C$. 
The moduli space of JS stable pairs on $\overline{X}$ is projective 
from the GIT construction 
by Le Potier \cite{LePotier}.
Since $P^{\mathrm{JS}}_n(X, d)$ is an open and closed subscheme of the 
JS moduli space on $\overline{X}$, 
it is also projective. 
\end{proof}
If $C$ is an elliptic curve, then $X$ is a CY 4-fold and 
the above lemma shows that the $\mathrm{DT}_4$
 virtual class in this case is the 
usual fundamental class: 
\begin{cor}\label{cor on JS vir clas}
Let $X=\oO_{\mathbb{P}^1}(-1,-1)\times E$ where $E$ is an elliptic curve. Then 
\begin{align*}
[P^{\mathrm{JS}}_n(X,d)]^{\mathrm{vir}}=[P^{\mathrm{JS}}_n(X,d)]\in H_{2n}(P^{\mathrm{JS}}_n(X,d),\mathbb{Z}),
\end{align*}
for certain choice of orientation.
\end{cor}
For a line bundle $L\in \Pic(X)$, we consider its tautological complex 
 \begin{align*}
L^{[n]}:=\dR\pi_{P\ast}(\mathbb{F}\otimes \pi_X^{\ast}L),
\end{align*}
where $\pi_X$, $\pi_P$ are projections from $X \times P^{\mathrm{JS}}_n(X,d)$
to corresponding factors and $\mathbb{I}^{\bullet}=\{\oO \to \mathbb{F}\}$ 
is the universal JS pair on $X \times P^{\mathrm{JS}}_n(X,d)$. 
By Lemma \ref{lem on smoothness of JS moduli}, we can consider 
the integration of their Euler classes. 
\begin{thm}\label{thm on integ for any C}
Let $X=\oO_{\mathbb{P}^1}(-1,-1)\times C$ for a smooth projective curve $C$.
We take $L=\pi^{\ast}L_C$
for a degree $a \geqslant 0$ line bundle $L_C$ on $C$. 
Then for any fixed $k\geqslant 1$, 
we have an identity of generating series
\begin{align*}
\sum_{d\geqslant 0,\, n=kd}&\int_{P^{\mathrm{JS}}_{n}(X,d)}e(L^{[n]})\,y^d=(1+y)^{ka}.
\end{align*}
\end{thm}
\begin{proof}
From the proof of Lemma \ref{lem on smoothness of JS moduli}, it is easy to see $H^1(X,L)=0$, so $L^{[n]}$ is a vector bundle. 
By the Grothendieck-Riemann-Roch formula, the integral depends only on the Chern character of $L_C$, so we may
assume $L_C=\oO_C(p_1+\ldots+p_a)$ with $p_i\neq p_j$ if $i\neq j$.
We denote by 
\begin{align*}
D_i=\oO_{\mathbb{P}^1}(-1,-1)\times \{p_i\}, \quad 
D=\coprod_{i=1}^a D_i. 
\end{align*}
Let $\sigma$ be the tautological section of $L^{[n]}$ 
as in Lemma~\ref{lem on tat sect}, 
whose zero locus is identified with 
\begin{align*}
P^{\mathrm{JS}}_n(D, d)=
\coprod_{\begin{subarray}{c}
d_1+d_2+\cdots+d_a=d  \\ n_i=k d_i \end{subarray}}P^{\mathrm{JS}}_{n_1}(D_1,d_1)\times \cdots \times P^{\mathrm{JS}}_{n_a}(D_a,d_a). 
\end{align*}
We have a morphism 
of vector bundles
\begin{align}\label{dsigma}
T_{P^{\mathrm{JS}}_{n}(X,d)}|_{P^{\mathrm{JS}}_{n}(D,d)}\stackrel{d\sigma}{\to} L^{[n]}|_{P^{\mathrm{JS}}_{n}(D,d)}, 
\end{align}
whose kernel is $T_{P^{\mathrm{JS}}_{n}(D,d)}$. We claim that 
\begin{align}\label{claim on coker}\mathrm{Coker}(d\sigma)\cong 
T_{P^{\mathrm{JS}}_{n}(D,d)}. \end{align}
Suppose that the above isomorphism holds. 
Then we have 
\begin{align*}
\int_{P^{\mathrm{JS}}_n(X,d)}e(L^{[n]})&=\int_{P^{\mathrm{JS}}_n(D,d)}e(\mathrm{Coker}(d\sigma))=\chi(P^{\mathrm{JS}}_n(D,d)).
\end{align*}
Note that
$P_{kd_i}^{\mathrm{JS}}(D_i, d_i)$ consists of JS stable pairs
on the resolved conifold $D_i$:
\begin{align*}
\oO_{D_i} \stackrel{s}{\to} \oO_{\mathbb{P}^1}(k-1)^{\oplus d_i}, \quad
s=(s_1, \ldots, s_{d_i})
\end{align*}
where $s_j \in H^0(\oO_{\mathbb{P}^1}(k-1))=\mathbb{C}^k$
are linearly independent. 
Therefore 
$P_{kd_i}^{\mathrm{JS}}(D_i, d_i)$
is isomorphic to the Grassmannian 
$\mathrm{Gr}(d_i, k)$. 
By the above identities, we have 
\begin{align*}
\int_{P^{\mathrm{JS}}_n(X,d)}e(L^{[n]})
&=\sum_{\begin{subarray}{c}
d_1+d_2+\cdots+d_a=d  \\ n_i=k d_i \end{subarray}}
\prod_{i=1}^a \chi(\mathrm{Gr}(d_i, k)) \\
&=\sum_{\begin{subarray}{c}
d_1+d_2+\cdots+d_a=d  \\ n_i=k d_i \end{subarray}}
\prod_{i=1}^a {k \choose d_i}. 
\end{align*}
Therefore we are left to prove \eqref{claim on coker}.
For simplicity, we prove it for $a=1$ case. 
Let $i \colon D \hookrightarrow X$ be the inclusion. 
Let us take JS stable pairs on $D$ and $X$:
 \begin{align*}
I_D=(\oO_D\to F)\in P^{\mathrm{JS}}_{n}(D,d), \quad 
I_X=(\oO_X\to i_{\ast}F)\in P^{\mathrm{JS}}_{n}(X,d),
\end{align*}
where $I_X$ is induced from $I_D$. 
Then we have the following description of
tangent spaces of JS stable pair moduli spaces
 \begin{align*}
 T_{P^{\mathrm{JS}}_{n}(D,d)}|_{I_D}=\Hom_D(I_D,F), \quad 
 T_{P^{\mathrm{JS}}_{n}(X,d)}|_{I_X}=\Hom_X(\oO_X\to i_{\ast}F,i_{\ast}F). 
\end{align*}
Note that the latter space is isomorphic to 
$\Hom_D(\oO_D\to \dL i^*i_*F,F)$ by the adjunction. 
We have a distinguished triangle
$$F\boxtimes \oO_{C}(-p)|_p[1]\to \dL i^*i_*F \to F, $$
which implies a distinguished triangle
\begin{align*}
F\boxtimes \oO_{C}(-p)|_p\to \dL i^*I_X \to I_D. 
\end{align*}
By applying $\RHom_D(-, F)$, we
obtain an exact sequence:
 \begin{align}\label{equ ne1}
\to \Hom_X(I_X,i_*F) \stackrel{\alpha}{\to} \Hom_D(F,F\boxtimes \oO_{C}(p)|_p)
\to \Ext^1_D(I_D,F). 
 \end{align}
Here we have $\Ext^1_D(I_D, F)=\Ext^2_D(F, F)=0$
since 
$H^1(F)=0$ and $F$ is a semistable sheaf on $D$. 
Therefore $\alpha$ is surjective. 
We also 
apply $\RHom_D(-,F\boxtimes  \oO_{C}(p)|_{p})$ to 
the distinguished triangle 
$I_D\to \oO_D\to F$ and get an exact sequence:
\begin{align}\label{equ ne20}\to \Hom_D(F,F\boxtimes \oO_{C}(p)|_{p}) 
\stackrel{\beta}{\to} H^0(F\boxtimes \oO_{C}(p)|_{p}) &\to \Hom_D(I_D,F\boxtimes \oO_{C}(p)|_{p}) \\
\notag
&\to \Ext^1_D(F,F\boxtimes \oO_{C}(p)|_{p})=0. 
\end{align}
The map $d\sigma$ in (\ref{dsigma}) is given by the composition 
\begin{align*}
\Hom_X(I_X,i_*F) \stackrel{\alpha}{\twoheadrightarrow}
 \Hom_D(F,F\boxtimes \oO_{C}(p)|_p) \stackrel{\beta}{\to}
 H^0(F\boxtimes \oO_{C}(p)|_{p}). 
\end{align*}
By \eqref{equ ne1}, \eqref{equ ne20}, we have 
\begin{align*}\mathrm{Coker}(d\sigma)|_{I_D}\cong \Hom_D(I_D,F\boxtimes 
\oO_{C}(p)|_{p})\cong \Hom_D(I_D,F), \end{align*}
which gives the tangent space $T_{P^{\mathrm{JS}}_{n}(D,d)}$ at $I_D$.
\end{proof}

By combining with Corollary \ref{cor on JS vir clas}, we
obtain the following result:
\begin{thm}\label{verify conj in JS for local P1 times E}
Let $E$ be an elliptic curve, $X=\oO_{\mathbb{P}^1}(-1,-1)\times E$ and 
$L=\pi^{\ast}L_E$ for a degree $a \geqslant 0$ line bundle on $E$. Then for any fixed $k\geqslant 1$, we have 
\begin{align*}
\sum_{d\geqslant 0,\,n=kd}\int_{[P^{\mathrm{JS}}_n(X,d)]^{\mathrm{vir}}}e(L^{[n]})\,q^ny^d=(1-(-q)^ky)^{ka}, 
\end{align*}
for a choice of orientation. In particular, Conjecture \ref{main conj} holds in JS chamber for $X$ and $L$. 
\end{thm}
\begin{proof}
For $L=\pi^{\ast}L_E$ ($a>0$), a curve class $\beta$ which satisfies $L\cdot \beta=0$ is a multiple of class $[\mathbb{P}^1]$. For such 
curve classes, genus one GV type invariants of $X$ are zero (ref. \cite[Lemma 2.16]{CMT2}). 
The genus zero GV invariants $n_{0,d}$ of $\oO_{\mathbb{P}^1}(-1,-1)$ are well-known to satisfy (e.g.~\cite[Proposition~4.3]{HST})
\begin{align*}
n_{0,1}=1, \quad n_{0,d>1}=0. 
\end{align*}
Therefore, Conjecture \ref{main conj} holds in JS chamber by adding sign $(-1)^{n+d}$ for $P^{\mathrm{JS}}_{n,d}(L)$.
\end{proof}


\subsection{Applications to local resolved conifold}\label{sect on local res conifold}
Let $X=\oO_{\mathbb{P}^1}(-1,-1,0)$ which is a toric CY 4-fold.
It admits a $T$-action, where $T$ is a CY torus
\begin{align}\label{repara T}T=\{t=(t_0,t_1,t_2,t_3)\in(\mathbb{C}^*)^4:\,t_0t_1t_2t_3=1\},
\end{align}
whose action is given by 
$t\cdot (x_0,x_1,x_2,x_3)=(t_0x_0,t_1x_1,t_2x_2,t_3x_3)$
in local coordinates of $X$. 
In this notation, the normal bundle of the zero section satisfies
$$N_{\mathbb{P}^1/X}=\oO_{\mathbb{P}^1}(-Z_{\infty})\otimes t_1^{-1}\oplus \oO_{\mathbb{P}^1}(-Z_{\infty})\otimes t_2^{-1}\oplus
\oO_{\mathbb{P}^1}\otimes t_3^{-1},$$
where $Z_{\infty}\in \mathbb{P}^1$ is the torus fixed point $[1:0]$.
The above torus action lifts to an action on the moduli space $P^{\mathrm{JS}}_n(X,d)$ of JS stable pairs.
By \cite[Proposition 2.1]{CT3} (see also \cite[Lemma 6.7]{CT1}), 
the $T$-fixed locus 
$P^{\mathrm{JS}}_n(X,d)^T$ consists of finitely many reduced points. Hence for any $I\in P^{\mathrm{JS}}_n(X,d)^T$, 
the tangent space 
\begin{align*}
T_{P^{\mathrm{JS}}_n(X,d)}|_{I}=\Ext^1(I,I)_0
\end{align*}
 has no $T$-fixed subspace and its equivariant Euler class is not zero.

We recall the following notion of square roots. 
\begin{defi}\label{def:sroot}
Let $K^{T}(pt)$ denote the $T$-equivariant $K$-theory of one point. A square root $V^{\frac{1}{2}}$ of $V\in K^{T}(pt)$
is an element in $K^{T}(pt)$ such that 
$$V^{\frac{1}{2}}+\overline{V^{\frac{1}{2}}}=V. $$
Here $\overline{(\cdot)}$ denotes the involution on $K^{T}(pt)$ induced by $\mathbb{Z}$-linearly extending the map 
$$t_0^{w_0}t_1^{w_1}t_2^{w_2}t_3^{w_3} \mapsto t_0^{-w_0}t_1^{-w_1}t_2^{-w_2}t_3^{-w_3}, $$
where $t_i$'s denote torus weights in notation \eqref{repara T}.
\end{defi}
For a $T$-equivariant pair $I=(\oO_X\stackrel{s}{\to} F)\in P^{\mathrm{JS}}_n(X,d)^T$, we choose the following square root:
\begin{align}\label{eq on choice of sqr root}\chi_X(I,I)_0^{\frac{1}{2}}:=-\chi_X(F)+\chi_Y(F,F)\in K^{T}(pt), \end{align}
where $Y=\oO_{\mathbb{P}^1}(-1,0)$ is a Fano 3-fold. Here $\chi_Y(F,F)$ makes sense as $F$ is semistable and so it is scheme theoretically supported on $Y$.
It is easy to see that (\ref{eq on choice of sqr root})
is a square root of $\chi_X(I, I)_0$ in the sense of 
Definition~\ref{def:sroot} by Serre duality and adjunction formula 
$$\chi_X(F,F)=\chi_Y(F,F)-\chi_Y(F,F\otimes K_Y).$$
\begin{defi}\label{defi of equi taut inv}
Let $X=\oO_{\mathbb{P}^1}(-1,-1,0)$ and consider a trivial $\mathbb{C}^*$-action on moduli spaces such that $e^m$ is a trivial line bundle with $\mathbb{C}^*$-equivariant weight $m$. 
We define (equivariant) tautological JS stable pair invariants by
\begin{align*}
P^{\mathrm{JS}}_{n,d}(e^m):=\sum_{\begin{subarray}{c}I=(\oO_X\to F) \in P^{\mathrm{JS}}_n(X,d)^{T} \end{subarray}}
(-1)^{d}e(\chi_X(I,I)^{\frac{1}{2}}_0+\chi_X(F)\otimes e^m)\in \frac{\mathbb{Q}(\lambda_0,\lambda_1,\lambda_2,\lambda_3,m)}{(\lambda_0+\lambda_1+\lambda_2+\lambda_3)}, 
\end{align*}
where the equivariant Euler class is taken with respect to $T\times \mathbb{C}^*$-action and $\lambda_i=e_{T}(t_i)$'s are equivariant parameters of $T$ \eqref{repara T}.
\end{defi}
\begin{rmk}
We use a different convention as \cite[Definition 2.3]{CT3} which uses $\chi_X(F)^{\vee}\otimes e^m$ instead of $\chi_X(F)\otimes e^m$. Their Euler classes differ by 
a sign $(-1)^n$ $($after replacing $m\to -m$$)$. Our choice of sign is consistent with the convention in \cite[Remark 2.4]{CT3}.
\end{rmk}
We can classify all torus fixed JS stable pairs and compute the (equivariant) tautological invariants explicitly.
\begin{lem}\label{T-fixed JS pairs}\emph{(\cite[Lemma 6.6]{CT1})}
Let $k\geqslant 0$, $n=d(k+1)$ and
$Z_0=[0:1]$, $Z_{\infty}=[1:0]$ are 
the torus fixed points of $\mathbb{P}^1$. 
Then 
a $T$-fixed JS stable pair 
$I=(\oO_X\stackrel{s}{\to}  F)\in P^{\mathrm{JS}}_{n}(X,d)^{T}$
is precisely of the form
\begin{align*}
F=\bigoplus_{i=0}^k
\oO_{\mathbb{P}^1}\big((k-i)Z_{\infty} +iZ_{0}\big)
\Big(\sum_{j=0}^{d_i-1} t_3^{j} \Big),
\end{align*}
for some $d_0,\ldots, d_k\geqslant 0$ with $\sum_{i=0}^k d_i=d$, and $s$ is given by a canonical section.
\end{lem}
\begin{prop}\label{JS taut invs}\emph{(\cite[Theorem 2.15]{CT3})}
Let $k\geqslant 0$, $n=d(k+1)$, then 
\begin{align*}
P^{\mathrm{JS}}_{n,d}(e^m)&=\frac{(-1)^{n}}{1!\,2!\,\cdots k!}\cdot
\sum_{\begin{subarray}{c}d_0+\cdots+d_k=d  \\  d_0,\ldots, d_k\geqslant 0 
\end{subarray}}\frac{1}{d_0!\cdots d_k!}\cdot \prod_{\begin{subarray}{c}i<j  \\  0\leqslant i,j \leqslant k \end{subarray}}\left((j-i)+(d_i-d_j)\frac{\lambda_3}{\lambda_0}\right) \\
&\times \prod_{i=0}^k\left(\prod_{\begin{subarray}{c}0\leqslant a\leqslant d_i-1  \\  -i\leqslant b\leqslant k-i  \end{subarray}}\left(\frac{-m}{\lambda_3}-a-b\frac{\lambda_0}{\lambda_3}\right)\cdot \prod_{\begin{subarray}{c}1\leqslant a\leqslant d_i  \\  1\leqslant b\leqslant k-i  \end{subarray}}\frac{1}{a+b\frac{\lambda_0}{\lambda_3}}\cdot
\prod_{\begin{subarray}{c}1\leqslant a\leqslant d_i  \\  1\leqslant b\leqslant i  \end{subarray}}\frac{1}{a-b\frac{\lambda_0}{\lambda_3}}\right).
\end{align*}
If $d\nmid n$, we have $P^{\mathrm{JS}}_{n,d}(e^m)=0$.
\end{prop}
The combinatorics involved above is complicated which we can not solve directly. Nevertheless, we relate it to the compact 
geometry and use Theorem \ref{thm on integ for any C} to solve it completely.
\begin{thm}\label{calc of JS equi inv}
Let $k\in\mathbb{Z}_{\geqslant1}$, then we have  
\begin{align}\label{exp JS equiv invs}\sum_{d\geqslant 0}&P^{\mathrm{JS}}_{kd, d}(e^m)y^d=(1-y)^{k\cdot\frac{m}{-\lambda_3}}. \end{align}
\end{thm}
\begin{proof}
	We apply Theorem \ref{thm on integ for any C} to $\overline{X}=\oO_{\mathbb{P}^1}(-1,-1)\times \mathbb{P}^1$. Then 
	for each fixed $k\geqslant 1$, we have 
	\begin{align}\label{form:compact}
		\sum_{d\geqslant 0, n=kd}\int_{P^{\mathrm{JS}}_n(\overline{X},d)}e(L^{[n]})\,y^d=(1+y)^{ak}, 
	\end{align}
	where $L=\pi^*\oO_{\mathbb{P}^1}(a)$ is the pull-back of a degree $a \geqslant 0$ line bundle on $\mathbb{P}^1$.
	The theorem will be proved by describing 
	the left hand side in terms of torus fixed loci
	using $T$-localization formula. 	
	
Let $Y=\oO_{\mathbb{P}^1}(-1,0)$ be a Fano 3-fold inside $X=\oO_{\mathbb{P}^1}(-1,-1,0)$. 
Any JS stable pair on $X$ is scheme theoretically supported on $Y$, so 
the adjunction gives an isomorphism
\begin{align*}
P^{\mathrm{JS}}_n(Y,d) \stackrel{\cong}{\to}
 P^{\mathrm{JS}}_n(X,d), \quad (s:\oO_Y\to F) \mapsto (s:\oO_X\to i_*F),  
\end{align*}
where $i\colon Y\hookrightarrow X$ is the zero section
of the projection $X \to Y$. 
Let $I_Y=(\oO_Y\to F)$ be a $T$-fixed JS stable pair on $Y$, 
and $I_X=(\oO_X \to i_{\ast}F)$
the corresponding 
$T$-fixed JS stable pair on $X$. 
We apply $\RHom_Y(-,F)$ to 
the distinguished triangle 
$I_Y\to \oO_Y\to F$ and get an exact sequence:
\begin{align}\label{equ ne2}0 \to \Hom_Y(F,F)& \to H^0(F)\to \Hom_Y(I_Y,F)\to \Ext^1_Y(F,F)\to H^1(F)=0. 
\end{align}
By Lemma \ref{lem on smoothness of JS moduli} and the 
adjunction, we have $\Ext^2_X(i_*F,i_*F)=\Ext^2_Y(F,F)=0$. 
Together with the exact sequence \eqref{equ ne2}, 
we have an identity in $K^T(pt)$:
\begin{align*}-\chi_X(i_*F)+\chi_Y(F,F) &=-H^0(F)+\Hom_Y(F,F)-\Ext^1_Y(F,F) \\
&=-\Hom_Y(I_Y,F). \end{align*} 
Therefore, we obtain an
identity for the square root \eqref{eq on choice of sqr root}
\begin{align*}
\chi_X(I_X,I_X)_0^{\frac{1}{2}}=-T_{P^{\mathrm{JS}}_n(X,d)}|_{I_X}. 
\end{align*}
The $T$-action \eqref{repara T} on $X$ and $\overline{X}$ induces actions on corresponding moduli spaces of JS stable pairs.
Note that $\overline{X}$ has an open cover given by 
\begin{align*}
\overline{X}=\overline{X}_0 \cup \overline{X}_{\infty}, \quad 
\overline{X}_0=\oO_{\mathbb{P}^1}(-1,-1)\times \mathbb{C}_0, \quad 
\overline{X}_{\infty}=\oO_{\mathbb{P}^1}(-1,-1)\times \mathbb{C}_{\infty}. 
\end{align*}
Since any $T$-fixed JS stable pairs on $\overline{X}$ is supported either on  $\overline{X}_0$ or $\overline{X}_{\infty}$, 
we have 
\begin{align*}
P^{\mathrm{JS}}_n(\overline{X},d)^T=\coprod_{\begin{subarray}{c}d_1+d_2=d \\  n_1+n_2=n \end{subarray}} 
P^{\mathrm{JS}}_{n_1}(\overline{X}_0,d_1)^T\times P^{\mathrm{JS}}_{n_2}(\overline{X}_{\infty},d_2)^T. 
\end{align*}
We choose a $T$-equivariant structure of 
$L$ as $\pi^*\oO_{\mathbb{P}^1}(aW_0)$ where $W_0$ denotes the torus 
fixed point $0\in \mathbb{P}^1$. 
By applying Atiyah-Bott localization formula \cite{AB}, we obtain
\begin{align*}
&\quad \,\, \int_{P^{\mathrm{JS}}_n(\overline{X},d)}e(L^{[n]})=\int_{P^{\mathrm{JS}}_n(\overline{X},d)^T}e_T(L^{[n]})  \\
&=\sum_{I_{\overline{X}}=(\oO_{\overline{X}}\to F)\in P^{\mathrm{JS}}_n(\overline{X},d)^T}
\frac{e_T(H^0(F\otimes L))}{e_T(T_{P^{\mathrm{JS}}_n(\overline{X},d)}|_{I_X})}
 \\ 
&=\sum_{\begin{subarray}{c}d_1+d_2=d \\  n_1+n_2=n \end{subarray}}
\sum_{\begin{subarray}{c}I_{\overline{X}_0}=(\oO_{\overline{X}_0}\to F_0)\in P^{\mathrm{JS}}_{n_1}(\overline{X}_{0},d_1)^T \\ 
I_{\overline{X}_{\infty}}=(\oO_{\overline{X}_{\infty}}\to F_{\infty})\in P^{\mathrm{JS}}_{n_2}(\overline{X}_{\infty},d_2)^T \end{subarray}}
\frac{e_T(H^0(F_{0}\boxtimes \oO_{\mathbb{P}^1}(aW_0)))}{e_T(T_{P^{\mathrm{JS}}_{n_1}(\overline{X}_{0},d_1)}|_{I_{\overline{X}_{0}}})}\cdot 
\frac{e_T(H^0(F_{\infty}))}{e_T(T_{P^{\mathrm{JS}}_{n_2}(\overline{X}_{\infty},d_2)}|_{I_{\overline{X}_{\infty}}})}  \\
&=\sum_{\begin{subarray}{c}I_{\overline{X}_0}=(\oO_{\overline{X}_0}\to F_0)\in P^{\mathrm{JS}}_{n}(\overline{X}_{0},d)^T  \end{subarray}}
\frac{e_T(H^0(F_{0}\boxtimes \oO_{\mathbb{P}^1}(aW_0)))}{e_T(T_{P^{\mathrm{JS}}_{n}(\overline{X}_{0},d)}|_{I_{\overline{X}_{0}}})} \\
&=\sum_{\begin{subarray}{c}I_{X}=(\oO_{X}\to F)\in P^{\mathrm{JS}}_{n}(X,d)^T  \end{subarray}}
e_T(H^0(F\boxtimes \oO_{\mathbb{P}^1}(aW_0)))\cdot e_T(\chi_X(I_X,I_X)_0^{\frac{1}{2}}), 
\end{align*}
where in the fourth equality, we use the fact that $e_T(H^0(F_{\infty}))=0$ as $H^0(F_{\infty})=\chi(F_{\infty})$ has weight zero component by Lemma \ref{T-fixed JS pairs}.
Together with (\ref{form:compact}), 
this implies that \eqref{exp JS equiv invs} holds when $m=-a\lambda_3$ for any $a\geqslant 0$. Therefore we are reduced to the following Lemma \ref{lem on matrix}.
\end{proof}
To complete the proof, we show:
\begin{lem}\label{lem on matrix}
The equality \eqref{exp JS equiv invs} holds if and only if it holds for $m=-a\lambda_3$ for any $a\geqslant 0$.
\end{lem}
\begin{proof}
By Definition \ref{defi of equi taut inv}, $P^{\mathrm{JS}}_{n,d}(e^m)$ is a polynomial in variable $m$ with coefficients to be rational 
functions of equivariant parameters $\lambda_0,\lambda_1,\lambda_2,\lambda_3$.  
By Proposition \ref{JS taut invs}, setting $\lambda_3=-1$ does make sense and we write
\begin{align*}
P^{\mathrm{JS}}_{n,d}(e^m)\big|_{\lambda_3=-1}=\sum_{i=0}^nf_i(\lambda_0,\lambda_1,\lambda_2)\,m^i, 
\end{align*}
where $f_i(\lambda_0,\lambda_1,\lambda_2)$'s are rational functions. 
By our assumption, it is 
independent of $\lambda_0, \lambda_1, \lambda_2$
when $m\in \mathbb{Z}_{\geqslant 0}$. So for any $l=0,1,2$, we have 
\begin{align}\label{equ on part derive}\sum_{i=0}^n\frac{\partial f_i(\lambda_0,\lambda_1,\lambda_2)}{\partial\lambda_l}\,m^i=0, \quad \forall\,\, m\in \mathbb{Z}_{\geqslant 0}. \end{align}
For any fixed $l=0,1,2$, consider the first $(n+1)$ equations of above, i.e.  
\begin{align*}\frac{\partial f_0(\lambda_0,\lambda_1,\lambda_2)}{\partial\lambda_l}\,m^0+\frac{\partial f_1(\lambda_0,\lambda_1,\lambda_2)}{\partial\lambda_l}\,m^1+\cdots+
\frac{\partial f_n(\lambda_0,\lambda_1,\lambda_2)}{\partial\lambda_l}\,m^n=0, \,\,\mathrm{where}\,\, m=0,1,\ldots,n. \end{align*}
We view these as linear equations on variables $\left\{\frac{\partial f_i(\lambda_0,\lambda_1,\lambda_2)}{\partial\lambda_l}\right\}_{i=0,1,\ldots,n}$. As the coefficient matrix is the $(n+1)\times (n+1)$ Vandermonde matrix $(a_{ij})=(i)^{j-1}$ whose determinant is nonzero. Therefore, we conclude
$$\frac{\partial f_i(\lambda_0,\lambda_1,\lambda_2)}{\partial\lambda_l}=0, \quad\mathrm{for}\,\,\mathrm{any}\,\, 0\leqslant i\leqslant n,\,\, l=0,1,2. $$ 
So $f_i(\lambda_0,\lambda_1,\lambda_2)=c_i$ is a constant for any $i$. We are left to show 
$$\sum_{i=0}^n c_i\,m^i=(-1)^d{(k+1)m \choose d}, $$
where ${x \choose d}:=x(x-1)\cdots (x-d+1)/d\,!$ for a formal variable $x$ and $d\in\mathbb{Z}_{\geqslant0}$.
By the invertibility of Vandermonde matrix, this holds for any $m$ as it holds for any $m\in \mathbb{Z}_{\geqslant 0}$.
\end{proof}
Using Theorem \ref{calc of JS equi inv}, we are able to solve our previous conjecture~\cite[Conjecture~6.10]{CT1}
 on equivariant JS stable pair invariants without insertions.
\begin{cor}\label{cor on inv no insert}
We have the identity
\begin{align*}
\sum_{\begin{subarray}{c}I \in P^{\mathrm{JS}}_n(X,d)^{T} \end{subarray}}e_{T}(\chi_X(I,I)^{\frac{1}{2}}_0)=
\left\{\begin{array}{rcl}\frac{1}{d\,!\,\lambda_3^d} \quad         &\mathrm{if} \,\,n=d, \\
& \\ 0   \, \, \,\, \,  \quad      & \,\,  \mathrm{otherwise}. 
\end{array} \right. 
\end{align*} 
Here $-\lambda_3$ is the equivariant parameter of $\oO_{\mathbb{P}^1}(0)$ in $X=\oO_{\mathbb{P}^1}(-1,-1,0)$. 
\end{cor} 
\begin{proof}
We apply the following insertion-free limit (see \cite[\S0.5]{CKM1}, \cite[Proposition~2.6]{CT3}):
\begin{align*}
&\quad\,  \lim_{\begin{subarray}{c}Q\, \mathrm{fixed} \\ m\to \infty \end{subarray}}\left(\sum_{\begin{subarray}{c}n=(k+1)d \\ d\geqslant 0 
\end{subarray}}P^{\mathrm{JS}}_{n,d}(e^m)\,q^n\big|_{Q=qm}\right) \\
&=\lim_{\begin{subarray}{c}m\to \infty \end{subarray}}\left(\sum_{\begin{subarray}{c}n=(k+1)d \\ d\geqslant 0 
\end{subarray}}\frac{P^{\mathrm{JS}}_{n,d}(e^m)}{m^n}\,Q^n\right)\\
&=\lim_{\begin{subarray}{c}m\to \infty \end{subarray}}\left(\sum_{\begin{subarray}{c}n=(k+1)d \\ d\geqslant 0 
\end{subarray}}Q^n\sum_{\begin{subarray}{c}I \in P^{\mathrm{JS}}_n(X,d)^{T} \end{subarray}}
e_{T}(\chi_X(I,I)^{\frac{1}{2}}_0)\cdot \frac{e_{T\times \mathbb{C}^*}(\chi_X(F)\otimes e^m)}{m^n}\right) \\
&=\lim_{\begin{subarray}{c}m\to \infty \end{subarray}}\left(\sum_{\begin{subarray}{c}n=(k+1)d \\ d\geqslant 0 
\end{subarray}}Q^n\sum_{\begin{subarray}{c}I \in P^{\mathrm{JS}}_n(X,d)^{T} \end{subarray}}
e_{T}(\chi_X(I,I)^{\frac{1}{2}}_0)\cdot \frac{e_{T\times \mathbb{C}^*}(H^0(X,F)\otimes e^m)}{m^n}\right) \\
&=\lim_{\begin{subarray}{c}m\to \infty \end{subarray}}\left(\sum_{\begin{subarray}{c}n=(k+1)d \\ d\geqslant 0 
\end{subarray}}Q^n\sum_{\begin{subarray}{c}I \in P^{\mathrm{JS}}_n(X,d)^{T} \end{subarray}}
e_{T}(\chi_X(I,I)^{\frac{1}{2}}_0)\cdot \frac{(m^n+\mathrm{l.o.t.})}{m^n}\right) \\
&=\sum_{\begin{subarray}{c}n=(k+1)d \\ d\geqslant 0 
\end{subarray}}Q^n\sum_{\begin{subarray}{c}I \in P^{\mathrm{JS}}_n(X,d)^{T} \end{subarray}}
e_{T}(\chi_X(I,I)^{\frac{1}{2}}_0),
\end{align*}
where `l.o.t.' means lower order terms of $m$ and we use $\chi_X(F)=n$ in the fourth equality. 

By Theorem \ref{calc of JS equi inv}, we have 
\begin{align*}
\lim_{\begin{subarray}{c}Q\, \mathrm{fixed} \\ m\to \infty \end{subarray}}\left(\sum_{\begin{subarray}{c}n=(k+1)d \\ d\geqslant 0 
\end{subarray}}P^{\mathrm{JS}}_{n,d}(e^m)\,q^n\big|_{Q=qm}\right) 
&=\lim_{\begin{subarray}{c}Q\, \mathrm{fixed} \\ m\to \infty \end{subarray}}\left(\left(1-q^{k+1}\right)^{(k+1)\frac{-m}{\lambda_3}}\Big|_{Q=qm}\right) \\
&=\lim_{\begin{subarray}{c} m\to \infty \end{subarray}}\left(\left(1-\frac{Q^{k+1}}{m^{k+1}}\right)^{(k+1)\frac{-m}{\lambda_3}}\right) \\
&= \left\{\begin{array}{rcl}\exp\left(\frac{Q}{\lambda_3}\right)  \quad         &\mathrm{if} \,\,k=0, \\
& \\ 1   \quad \quad  \quad     &   \mathrm{otherwise}. 
\end{array} \right. 
\end{align*}
Therefore we are done.
\end{proof}

\appendix
\section{Vanishing for local elliptic curves}\label{sect on vanishing}

In this section, we take 
$$X=\mathrm{Tot}_E(L_1\oplus L_2 \oplus L_3)$$ 
to be the total space of direct sum of three general degree zero line bundles $L_1,L_2,L_3$ on an elliptic curve $E$ satisfying 
$L_1 \otimes L_2 \otimes L_3=\oO_E$. 

Let $T=(\mathbb{C}^{\ast})^{ 3}$ be the three dimensional complex torus
which acts on the fibers of $X\to E$. Its restriction to the subtorus
\begin{align*}
T_0=\{t_1 t_2 t_3=1\} \subset T
\end{align*}
preserves the holomorphic volume form of $X$.
Let $\bullet$ be the point $\Spec \mathbb{C}$ with trivial $T$-action, 
$\mathbb{C} \otimes t_i$ be the one dimensional $T$-representation with weight $1$,
and $\lambda_i \in H_T^{\ast}(\bullet)$ be its first Chern class.
They are generators of equivariant cohomology rings:
\begin{align*} 
H_{T}^{\ast}(\bullet)=\mathbb{C}[\lambda_1, \lambda_2, \lambda_3], \quad \
H_{T_0}^{\ast}(\bullet)=\frac{\mathbb{C}[\lambda_1, \lambda_2, \lambda_3]}{(\lambda_1+\lambda_2+\lambda_3)} \cong \mathbb{C}[\lambda_1, \lambda_2].
\end{align*}
We consider the induced $T_0$-action on the moduli space $P_n(X,2)$ of PT stable pairs $(F,s)$ with $[F]=2\,[E]$ and $\chi(F)=n>0$. 
By \cite[Section 4.2]{CMT2}, $T_0$-fixed loci are described as follows:
\begin{align*} 
&\quad\,\, P_n(X,2)^{T_0}\\
&=\Big\{(F,s):\,\pi_*F=\oO_E(Z_0) \oplus \oO_E(Z_1)\otimes L_i^{-1}\cdot t_i^{-1},\,\,i=1,2,3 \,\,\mathrm{with} \,\,\mathrm{effective\,\, divisors} \,\, Z_0\subseteq Z_1  \Big\} \\
&\cong\, \coprod_{i=1}^3
\coprod_{\begin{subarray}{c}
n_0 \leqslant n_1 \\
n_0+n_1=n 
\end{subarray}}
\mathrm{Sym}^{n_0}(E)\times \mathrm{Sym}^{(n_1-n_0)}(E).
\end{align*}
We will define its equivariant tautological invariant by the localization formula as in \cite[Section 4.3]{CMT2} and prove its vanishing. 
First note that when $n_0>0$, we have 
$$H^1(X,F)=H^1(E,\pi_*F)=0. $$
Therefore the tautological complex $\oO^{[n]}$ is a bundle when restricted to $T_0$-fixed loci with $n_0>0$. Then the vanishing of invariant 
of such component follows from the existence of nowhere zero section as in Proposition \ref{lem on vec bdl on JS}. 
So we may assume $n_0=0$ and $n_1=n>0$. 

For $I=(\oO_X\to F)\in P_n(X,2)^{T_0}$ with $\pi_*F=\oO_E \oplus
\oO_E(Z)\otimes L_1^{-1}\cdot t_1^{-1}$ (similarly we can consider other two cases via replacing $L_1^{-1}\cdot t_1^{-1}$ by $L_2^{-1}\cdot t_2^{-1}$ and $L_3^{-1}\cdot t_3^{-1}$), we define 
\begin{align*} 
\chi_X(I,I)_0^{\frac{1}{2}}&:=-\chi_X(F)+\chi_X(F,F)^{\frac{1}{2}} \\
&:=-\chi_E(\pi_*F)+\chi_X(\oO_E,\oO_E)+\chi_X(\oO_E,\oO_E(Z)\otimes L_1^{-1})\cdot t_1^{-1}.
\end{align*}
By adjunction, for $\alpha,\beta\in K(E)$ and zero section $i:E\hookrightarrow X$, we have 
\begin{align*} 
\chi_X(i_*\alpha,i_*\beta)=\chi_E\left(\alpha,(N^{\vee}-N)\otimes \beta\right),
\end{align*}
where $N$ is the normal bundle of $i \colon E \hookrightarrow X$
\begin{align*}
N=(L_1\otimes t_1)\oplus (L_2\otimes t_2)\oplus (L_3\otimes t_3).
\end{align*}
From this, it is easy to see $\chi_X(\oO_E,\oO_E)=0$ and 
\begin{align*}
\chi_X(\oO_E,\oO_E(Z)\otimes L_1^{-1})\cdot t_1^{-1}=\chi_E(\oO_E(Z))\otimes 
\big(t_1^{-2}+t_1^{-1}t_2^{-1}+t_1^{-1}t_3^{-1}-t_1^{-1}t_2-t_1^{-1}t_3-1 \big).\end{align*}
To sum up, 
the fixed part satisfies  
\begin{align*}
\dim_{\mathbb{C}}\left(-\chi_X(I,I)_0^{\frac{1}{2}}\right)^{\mathrm{fix}}=\dim_{\mathbb{C}}(\chi_E(\oO_E(Z)))=n=\dim_{\mathbb{C}}(\mathrm{Sym}^{n}(E)).
\end{align*}
Moreover, it is straightforward to check the $T_0$-fixed part of the Zariski tangent space $\Hom_X(I,F)$ has dimension $n$.
So the virtual class associated with this component of $P_n(X,2)^{T_0}$ is its usual fundamental class.
As for the movable part, we have 
\begin{align*}
\left(\chi_X(I,I)_0^{\frac{1}{2}}\right)^{\mathrm{mov}}=\chi_E(\oO_E(Z))
\otimes 
\big(-t_1^{-1}+t_1^{-2}+t_1^{-1}t_2^{-1}+t_1^{-1}t_3^{-1}-t_1^{-1}t_2-t_1^{-1}t_3 \big).
\end{align*} 
Therefore the 
virtual normal bundle $N^{\mathrm{vir}}$ is given by 
\begin{align*}
-N^{\mathrm{vir}}:=\dR \pi_{M*}\oO_{\mathrm{Sym}^{n}(E)\times E}(\mathcal{Z})\otimes 
\big(-t_1^{-1}+t_1^{-2}+t_1^{-1}t_2^{-1}+t_1^{-1}t_3^{-1}-t_1^{-1}t_2-t_1^{-1}t_3 \big), 
\end{align*}
where $\pi_M \colon \mathrm{Sym}^{n}(E)\times E\to \mathrm{Sym}^{n}(E)$ is the projection map and 
\begin{align*}
\mathcal{Z}\hookrightarrow \mathrm{Sym}^{n}(E)\times E
\end{align*}
denotes the universal divisor. The tautological complex when restricted to this $T_0$-fixed loci is 
\begin{align*}
\oO^{[n]}\big|_{\mathrm{Sym}^{n}(E)}&=
\dR \pi_{M\ast}(\oO_{\mathrm{Sym}^n(E) \times E} \oplus 
\oO_{\mathrm{Sym}^{n}(E)\times E}(\mathcal{Z})\boxtimes (L_1 \cdot t_1^{-1}))
 \\
&=\dR \pi_{M\ast}\oO_{\mathrm{Sym}^{n}(E)\times E}(\mathcal{Z})
\otimes t_1^{-1},
\end{align*}
where the second identity holds in $K(\mathrm{Sym}^n(E))$. 
We define tautological stable pair invariant by 
\begin{align*}
P_{n,2}(\oO):=\int_{\mathrm{Sym}^{n}(E)}
\frac{e\left(\oO^{[n]}\big|_{\mathrm{Sym}^{n}(E)}\right)}
{e\left(N^{\mathrm{vir}}\right)} \in 
\frac{\mathbb{Q}(\lambda_1, \lambda_2, \lambda_3)}
{(\lambda_1+\lambda_2+\lambda_3)}. 
\end{align*}
\begin{lem}\label{lem on comb formula of local elliptic curve}
If $n>0$, we have 
\begin{align}\label{equ of P_n,2}P_{n,2}(\oO)=\int_{\mathrm{Sym}^{n}(E)}\frac{(\xi-2\lambda_1)^{n-1}(\xi'-2\lambda_1)(\xi+\lambda_2)^{n-1}(\xi'+\lambda_2)(\xi+\lambda_3)^{n-1}(\xi'+\lambda_3)}{(\xi-\lambda_1+\lambda_2)^{n-1}(\xi'-\lambda_1+\lambda_2)(\xi-\lambda_1+\lambda_3)^{n-1}(\xi'-\lambda_1+\lambda_3)}. \end{align}
Here $\xi:=c_1(\oO_{\mathrm{Sym}^{n}(E)}(1))$ is for the Abel-Jacobi map 
\begin{align*}
\mathrm{AJ} \colon 
 \mathrm{Sym}^{n}(E)\to \mathrm{Pic}^{n}(E), \quad Z \mapsto \oO_E(Z), 
\end{align*}
$f$ is the fiber class of $\mathrm{AJ}$ and $\xi':=\xi+(n-1)f$.
\end{lem}
\begin{proof}
The line bundle associated with the universal divisor 
is given by   
\begin{align}\label{equ relate line bundles}
\oO_{\mathrm{Sym}^{n}(E)\times E}(\mathcal{Z})\cong (\mathrm{AJ}\times \mathrm{id}_E)^*\mathcal{L} \otimes \pi_M^*\oO_{\mathrm{Sym}^{n}(E)}(1), \end{align}
where $\mathcal{L}\to \Pic^n(E)\times E$ denotes the universal line bundle. 
Here note that $\oO_{\mathrm{Sym}^n(E)}(1)$ is the 
relative tautological line bundle of the 
projective bundle:  
$$\mathrm{Sym}^{n}(E)\cong \mathrm{Proj}\,\mathrm{Sym}\left((\pi_{M*}\mathcal{L})^{\vee}\right)\stackrel{\mathrm{AJ}}{\to}\Pic^n(E). $$
Let $0\in E$ denote a base point, under the identification, 
\begin{align*}
E\stackrel{\cong}{\to} \Pic^n(E), \quad x\mapsto \oO_E(x+(n-1)\,0), 
\end{align*}
we can write the universal line bundle $\lL$ as  
\begin{align*}
\mathcal{L}\cong \oO_{E\times E}(\Delta)\otimes \pi_E^*\oO_E((n-1)\,0), 
\end{align*}
where $\Delta\subset E\times E$ is the diagonal. It fits into the exact sequence\begin{align*}
0\to \oO_{E\times E}\otimes \pi_E^*\oO_E((n-1)\,0) \to \mathcal{L} \to \Delta_*\oO_E((n-1)\,0)\to 0. 
\end{align*}
Applying $\dR\pi_{M*}$, we obtain an identity in $K(E)$:
\begin{align*}
\dR\pi_{M*}\mathcal{L}=\oO_E^{\oplus (n-1)}+\oO_E((n-1)\,0).
\end{align*}
By \eqref{equ relate line bundles}, as elements in
$K(\mathrm{Sym}^n(E))$, we have 
\begin{align*}
\dR\pi_{M*}\oO_{\mathrm{Sym}^{n}(E)\times E}(\mathcal{Z})&=
\dR\pi_{M*}(\mathrm{AJ}\times \mathrm{id}_E)^*\mathcal{L} \otimes\oO_{\mathrm{Sym}^{n}(E)}(1) \\
&=\mathrm{AJ}^*\left(\dR\pi_{M*}\mathcal{L}\right)\otimes\oO_{\mathrm{Sym}^{n}(E)}(1) \\
&=\left(\oO_{\mathrm{Sym}^{n}E}^{\oplus (n-1)}+\mathrm{AJ}^*\oO_E((n-1)\,0)\right)\otimes \oO_{\mathrm{Sym}^{n}(E)}(1).
\end{align*}
Therefore we obtain 
\begin{align*}
& \quad \,\, e\left(\dR\pi_{M*}\oO_{\mathrm{Sym}^{n}(E)\times E}(\mathcal{Z})\otimes t_1^{i}t_2^{j}t_3^{k}\right) \\
&=e\left(\oO_{\mathrm{Sym}^{n}(E)}(1)^{\oplus (n-1)}\otimes t_1^{i}t_2^{j}t_3^{k}\right)\cdot 
e\left(\mathrm{AJ}^*\oO_E((n-1)\,0)\otimes \oO_{\mathrm{Sym}^{n}(E)}(1)\otimes t_1^{i}t_2^{j}t_3^{k}\right) \\
&=(\xi+i\lambda_1+j\lambda_2+k\lambda_3)^{n-1}\cdot \big((n-1)f+\xi+i\lambda_1+j\lambda_2+k\lambda_3 \big),
\end{align*}
where $\lambda_i=c_1^T(t_i)$, $\xi=c_1(\oO_{\mathrm{Sym}^{n}(E)}(1))$ and $f$ denotes the fiber class of the 
map $\mathrm{AJ} \colon \mathrm{Sym}^{n}(E)\to E$. 
Finally by using the relation $\lambda_1+\lambda_2+\lambda_3=0$, the desired formula holds.
\end{proof}
\begin{thm}
Let $X=\mathrm{Tot}_E(L_1\oplus L_2 \oplus L_3)$ be as above. If $n>0$, then $P_{n,2}(\oO)=0$, i.e. Conjecture \ref{main conj} holds in PT chamber for
 $L=\oO_X$
and $\beta=2\,[E]$.
\end{thm}
\begin{proof}
We are left to prove the vanishing of the formula obtained in Lemma \ref{lem on comb formula of local elliptic curve}. Let 
\begin{align}\label{equ on poly f}f(\xi):=\frac{(\xi-2\lambda_1)(\xi+\lambda_2)(\xi+\lambda_3)}{(\xi-\lambda_1+\lambda_2)(\xi-\lambda_1+\lambda_3)}. \end{align}
By the vanishing $\xi^i=0$ for $i>n$, we know $f(\xi)$ is a polynomial with coefficients to be rational functions of $\lambda_1, \lambda_2, \lambda_3$.
We claim for any such polynomial (not necessarily as \eqref{equ on poly f}), we have 
\begin{align*}
\int_{\mathrm{Sym}^{n}(E)}f(\xi)^{n-1}f(\xi')=0.
\end{align*}
Note that 
$$\xi^n=1-n, \quad \xi^{n-1}\cdot f=1, \quad f^2=0,$$  
which implies 
\begin{align}\label{equat on xi}\xi^a\cdot (\xi')^{n-a}=(n-1)(n-a-1), \quad \forall\,\, 0\leqslant a\leqslant n. \end{align}
For a polynomial $f(x)=\sum_{0\leqslant i\leqslant n}a_ix^i$ with $a_i\in \mathbb{Q}(\lambda_1,\lambda_2,\lambda_3)$, we have 
\begin{align}\label{equ comp1}
\nonumber
\left(f(\xi)^{n-1}f(\xi')\right)_{\mathrm{deg}=n}&=\sum_{\begin{subarray}{c}i_1+\cdots+i_n=n \\ i_1,\ldots,i_n\geqslant 0 \end{subarray}}a_{i_1}\cdots a_{i_n}\left(\xi^{i_1+\cdots +i_{n-1}}\cdot \xi^{' i_n}\right) \\  \nonumber
&=(n-1)\sum_{\begin{subarray}{c}i_1+\cdots+i_n=n \\ i_1,\ldots,i_n\geqslant 0 \end{subarray}}a_{i_1}\cdots a_{i_n}\cdot (i_{n}-1) \\  
&=(n-1)\sum_{k=0}^{n}(k-1)\,a_k\cdot \sum_{\begin{subarray}{c}i_1+\cdots+i_{n-1}=n-k \\ i_1,\ldots,i_{n-1}\geqslant 0 \end{subarray}}a_{i_1}\cdots a_{i_{n-1}}.
\end{align}
To prove this is zero, it is enough to show for any collection $P_0,P_1,\ldots,P_n\in \mathbb{Z}_{\geqslant 0}$ of $(n+1)$ numbers such that 
\begin{align}\label{equ label1}
&P_0+ P_1+\cdots+ P_n=n, \\
\label{equ label2} 
&0\cdot P_0+1\cdot P_1+\cdots+n\cdot P_n=n, 
\end{align}
the coefficient of term $a_0^{P_0}a_1^{P_1}\cdots a_n^{P_n}$ appearing in \eqref{equ comp1} is zero.
We denote 
$$S(P_0,P_1,\ldots,P_n)\in \mathbb{Z}_{\geqslant 0}$$
to be the number of collections of $n$ nonnegative integers which up to permutations are of form  
$$(\underbrace{0,\ldots,0}_{P_0},\underbrace{1,\ldots,1}_{P_1},\ldots,\underbrace{n,\ldots,n}_{P_n}).  $$
It is elementary to see 
$$S(P_0,P_1,\ldots,P_n)=\frac{n!}{P_0!\,P_1!\cdots P_n!}, $$
and 
\begin{align*}
\sum_{k=0}^{n}(k-1)\,a_k\cdot \sum_{\begin{subarray}{c}i_1+\cdots+i_{n-1}=n-k \\ i_1,\ldots,i_{n-1}\geqslant 0 \end{subarray}}a_{i_1}\cdots a_{i_{n-1}}&=
\sum_{k=0}^{n}(k-1)S(P_0,\ldots,P_{k-1},P_k-1,P_{k+1},\ldots,P_n) \\
&=\sum_{k=0}^{n}\frac{(k-1)\,n!}{P_0!\cdots\,P_{k-1}!\,(P_{k}-1)!\,P_{k+1}!\cdots P_n!}
\end{align*}
By multiplying $P_0!\cdots P_n!$, the above expression is zero if and only if 
$$\sum_{k=0}^n(k-1)P_k=0, $$
which follows from subtracting \eqref{equ label2} by \eqref{equ label1}.
\end{proof}
\begin{rmk}
Assuming our equivariant invariants equal to global invariants defined by virtual classes, the vanishing follows from \eqref{equat on xi} by taking non-equivariant limit $\lambda_1,\lambda_2,\lambda_3\to 0$ in \eqref{equ of P_n,2}. 
\end{rmk}

\providecommand{\bysame}{\leavevmode\hbox to3em{\hrulefill}\thinspace}
\providecommand{\MR}{\relax\ifhmode\unskip\space\fi MR }
\providecommand{\MRhref}[2]{%
 \href{http://www.ams.org/mathscinet-getitem?mr=#1}{#2}}
\providecommand{\href}[2]{#2}

\end{document}